\documentclass[11pt]{article}

\usepackage[utf8]{inputenc}
\usepackage{authblk}
\usepackage[hidelinks,bookmarksnumbered]{hyperref}
\usepackage{enumitem}

\usepackage[a4paper,left=25mm,right=25mm,top=18mm,bottom=19mm]{geometry}

\usepackage{array}

\usepackage{tikz}
\usepackage{tikz-cd}
\usetikzlibrary{arrows.meta}
\tikzcdset{arrow style=tikz,
diagrams={>={Stealth[round,length=4pt,width=6pt,inset=3pt]}}}

\usepackage{amsthm}

\usepackage{amsmath}
\usepackage{amssymb}
\usepackage{amsfonts}
\usepackage[T1]{fontenc}
\usepackage{xfrac}
\usepackage{color}
\usepackage{todonotes}
\usepackage[english]{babel}
\usepackage{mathtools}

\newtheorem{theorem}{Theorem}
\newtheorem{definition}[theorem]{Definition}
\newtheorem{conjecture}[theorem]{Conjecture}
\newtheorem{lemma}[theorem]{Lemma}

\newtheorem{proposition}[theorem]{Proposition}

\newtheorem{example}[theorem]{Example}
\newtheorem{remark}[theorem]{Remark}
\newtheorem{corollary}[theorem]{Corollary}

\newtheorem*{theorem*}{Theorem}

\DeclareMathOperator{\ord}{ord}

\DeclareMathOperator{\Jac}{Jac}
\DeclareMathOperator{\fpt}{fpt}
\DeclareMathOperator{\lct}{lct}

\newcommand{\kx}{k[[\underline{x}]]}
\newcommand{\equivr}{\overset{\mathcal{R}}{\sim}}
\newcommand{\equivk}{\overset{\mathcal{K}}{\sim}}

\numberwithin{theorem}{section}

\title{Bounds on the $F$-Pure Threshold of Isolated Hypersurface Singularities}
\author{Yotam Svoray\textsuperscript{\protect\hyperlink{aff1}{1},\protect\hyperlink{aff2}{2}}}
\date{}

\AtEndDocument{\bigskip{\footnotesize%
  \noindent\hypertarget{aff1}{}\textsuperscript{1}\textsc{Department of Mathematics, University of Utah, UT 84112} \par
  \noindent\hypertarget{aff2}{}\textsuperscript{2}\textsc{Department of Mathematics, Ben Gurion University of the Negev, Be’er
Sheva, Israel} \par
  \noindent\textit{E-mail address}: \texttt{svoray@math.utah.edu}
}}

\begin{document}

\maketitle

\begin{abstract}
We obtain bounds for the $F$-pure threshold of isolated hypersurface singularities over an algebraically closed field of positive characteristic $p$ in terms of classical singularity invariants. For an irreducible plane curve whose value semigroup has first two generators $\beta_0,\beta_1$, we show that the $F$-pure threshold is at most $\frac{1}{\beta_0}+\frac{1}{\beta_1}$ and differs from it by less than $\frac{1}{\gcd(\beta_0,\beta_1)p}$, with equality when certain base-$p$ digits add without carrying, and we give an example showing that higher weighted-order terms can change the $F$-pure threshold. We also conjecture a lower bound in terms of the Milnor number, and prove it for semi-quasi-homogeneous singularities, for all plane curves, and whenever the number of variables exceeds $p$. In particular, it holds in characteristic $2$. As applications, we reprove Igusa's formula for the log canonical threshold of a complex plane branch (with integer coefficients), prove the conjecture of Musta\c{t}\u{a}, Takagi and Watanabe for such branches, and bound the log canonical threshold of complex isolated hypersurface singularities.
\end{abstract}

\section{Introduction}

Our goal is to use techniques from the study of isolated hypersurface singularities to bound and analyze the $F$-pure threshold (see Definition~\ref{def:Fpt}) of such singularities over an algebraically closed field $k$ of characteristic $p>0$. Our basic tool is the comparison $\fpt(\operatorname{in}_w(f))\leq\fpt(f)\leq\min\{1,\frac{w_1+\cdots+w_n}{\ord_w(f)}\}$ for $f\in\mathfrak m\subset k[[x_1,\dots,x_n]]$ and positive weights $w$ (Proposition~\ref{prop:weighted_fpt}). Specifically, we present three main results. First, for an irreducible plane curve, we bound the $F$-pure threshold in terms of the first two elements $\beta_0,\beta_1$ of the minimal generating sequence of its semigroup of values (see Definition~\ref{def:semigroup}). The upper bound is inspired by the analogous formula for the log canonical threshold proved by Igusa in~\cite{igusa1977first} (see Remark~\ref{rem:igusa_lct}), and the lower bound shows that the $F$-pure threshold is always close to it. We use the base-$p$ digit notation of Definition~\ref{def:basep}:

\begin{theorem}\label{thm:semigr_fpt}
Let $f\in\mathfrak{m}^2\subset k[[x,y]]$ be irreducible, let $\beta_0,\beta_1$ be the first two elements of the minimal generating sequence of $\Gamma(f)$, and set $g=\gcd(\beta_0,\beta_1)$, $a=\frac{\beta_0}{g}$, $c=\frac{\beta_1}{g}$ and $L=L_p(\frac1a,\frac1c)$. Then $2\leq a<c$, $L\geq1$, and:
\begin{enumerate}
    \item\label{cor:beta_bound} $\frac{1}{g}\fpt(y^a+x^c)\leq\fpt(f)\leq\frac{1}{\beta_0}+\frac{1}{\beta_1}$.
    \item\label{cor:beta_gap} If $L=\infty$, that is, if the base-$p$ digits of $\frac1a$ and $\frac1c$ add without carrying, then $\fpt(f)=\frac{1}{\beta_0}+\frac{1}{\beta_1}$. If $L<\infty$, then
    \[
        0\leq\frac{1}{\beta_0}+\frac{1}{\beta_1}-\fpt(f)\leq\frac{1-\theta_L(\frac1a)-\theta_L(\frac1c)}{gp^L}<\frac{1}{gp^L}.
    \]
    \item\label{cor:beta_equality} If $p\equiv1\pmod{\frac{\beta_0\beta_1}{g^2}}$, then $\fpt(f)=\frac{1}{\beta_0}+\frac{1}{\beta_1}$.
    \item\label{cor:beta_tame} If $p\nmid\frac{\beta_0\beta_1}{g^2}$, then $\fpt(f)\geq\frac{1}{\beta_0}+\frac{1}{\beta_1}-\frac{1}{gp}\left(1-\frac1a-\frac1c\right)$.
\end{enumerate}
\end{theorem}

In particular, $\fpt(f)>\frac1{\beta_0}+\frac1{\beta_1}-\frac{1}{gp}$. For $g=1$, Part~\ref{cor:beta_bound} reads $\fpt(x^{\beta_0}+y^{\beta_1})\leq\fpt(f)\leq\frac1{\beta_0}+\frac1{\beta_1}$, and Example~\ref{ex:higher_terms_change_fpt} shows that its lower bound can be strict, so in positive characteristic the $F$-pure threshold of a branch is not determined by $\beta_0$ and $\beta_1$.\\

Second, we show how the Milnor number can be used to bound the $F$-pure threshold. For semi-quasi-homogeneous $f$, the weighted Jacobian estimates of~\cite{hernandez2016f,muller2018f}, combined with a numerical criterion for an isolated weighted initial form (Lemma~\ref{lem:milnor_min}), give a lower bound in terms of $n$, $p$ and $\mu(f)$ (Proposition~\ref{prop:sqh_milnor}). This bound is attained both by semi-quasi-homogeneous singularities and by singularities that are not semi-quasi-homogeneous (Example~\ref{ex:sharp}), and we conjecture that it always holds:

\begin{conjecture}\label{conj:milnor}
If $n\geq2$ and $f\in\mathfrak{m}^2$ satisfies $\mu(f)<\infty$, then $\fpt(f)\geq\lambda-\frac{n-1-\lambda}{p}$, where $\lambda=\min\{1,\frac{n}{1+\mu(f)^{1/n}}\}$.
\end{conjecture}

\begin{theorem}\label{thm:conj_cases}
Conjecture~\ref{conj:milnor} holds for $f\in\mathfrak{m}^2$ with $n\geq2$ and $\mu(f)<\infty$ in each of the following cases:
\begin{enumerate}
    \item\label{conjcase:sqh} After a change of coordinates, $\mu(\operatorname{in}_w(f))<\infty$ for some $w\in\mathbb{N}_{>0}^n$.
    \item\label{conjcase:plane} $n=2$.
    \item\label{conjcase:large} $n\geq p+1$.
\end{enumerate}
In particular, if $p=2$, the conjecture holds in every dimension.
\end{theorem}

The first case follows from Proposition~\ref{prop:sqh_milnor}, and the third case is elementary (Proposition~\ref{prop:conj_large_n}). For irreducible plane curves, Proposition~\ref{prop:conj_branches} gives an elementary proof of the second case, with a sharper bound, unless $p$ divides $\frac{\beta_1}{\gcd(\beta_0,\beta_1)}$. The general case is proved in Section~\ref{sec:plane}. It combines a positive-characteristic version of Varchenko's adapted coordinates with a description of $F$-thresholds of weighted binary forms by critical vectors, following Hern\'andez and Teixeira~\cite{hernandez2017f} (Section~\ref{sec:critical}). By reduction modulo $p$, it gives $\lct(f)\geq\min\{1,\frac{2}{1+\sqrt{\mu(f)}}\}$ for every $f\in\mathfrak{m}^2\subset\mathbb{Z}[[x,y]]$ with $\mu(f)<\infty$ (see Remark~\ref{rem:conj_consequences}).\\

Third, we use these results to obtain bounds for the log canonical threshold (see Definition~\ref{def:lct}) of a power series with integer coefficients (Corollary~\ref{corB}), and the conjecture of Musta\c{t}\u{a}, Takagi and Watanabe for irreducible plane curves (see Remark~\ref{rem:MTW_open_general}) in the following sharp form, which also reproves Igusa's formula:

\begin{corollary}\label{cor:MTW}
Let $f\in\mathfrak{m}^2\subset\mathbb{Z}[[x,y]]$ be irreducible in $\mathbb{C}[[x,y]]$, and let $\beta_0,\beta_1,g,a,c$ be as in Theorem~\ref{thm:semigr_fpt}, computed over $\mathbb{C}$. Then $\lct(f)=\frac1{\beta_0}+\frac1{\beta_1}$, and $\lct(f)-\frac{1}{gp}<\fpt(f_p)\leq\lct(f)$ for every prime $p\gg0$, with equality on the right whenever the base-$p$ digits of $\frac1a$ and $\frac1c$ add without carrying. In particular, $\fpt(f_p)=\lct(f)$ for every prime $p\gg0$ with $p\equiv1\pmod{ac}$.
\end{corollary}

Along the way, we bound the $F$-pure threshold in terms of the Milnor and Tjurina numbers (Propositions~\ref{prop:milnor_tjurina} and~\ref{prop:tau_upper}), and we show that $\fpt(f)\geq\frac{1}{\ord(f)}+\frac{1}{\mu(f)+\tau(f)}$ for every $f\in\mathfrak{m}^2\subset k[[x,y]]$ with $\tau(f)<\infty$ and $p\nmid\ord(f)$ (Proposition~\ref{prop:MT_curves}). We also bound it in terms of the tangent cone (Proposition~\ref{prop:tangent_cone}), prove a Thom--Sebastiani type lower bound (Lemma~\ref{lem:thom_sebastiani}), and record an elementary bound on the Brian{\c{c}}on--Skoda exponent (Remark~\ref{rem:ebs}).\\

\noindent\textbf{Acknowledgments.} This work was partially done as part of the author's PhD thesis under the guidance of Karl Schwede, and we wish to thank him for his guidance, help, and support. We wish to thank Patricio Almir\'{o}n, Benjamin Baily, Gari Chua, Ilya Smirnov, and Jack Jeffries for their input on the ideas presented in this paper. In addition, we thank Gari Chua for his input and for reviewing several drafts of this note. The author was partially supported by NSF grant DMS-2101800. The author also wishes to thank Yotam Hendel and Ben-Gurion University of the Negev for hospitality while working on the different drafts of  this project. 

\section{$F$-Pure Threshold of Isolated Singularities}\label{chap:F-stuff}

In this section, we relate the $F$-pure threshold of an isolated singularity to its weighted initial forms and to its Milnor and Tjurina numbers, and we prove the semi-quasi-homogeneous case of Conjecture~\ref{conj:milnor}. Several previous results compute and study the $F$-pure threshold of isolated singularities, especially for homogeneous and quasi-homogeneous polynomials with isolated singularities, as in~\cite{hernandez2016f,muller2018f,canton2016behavior,kadyrsizova2022lower}. For background on $F$-pure thresholds, we refer to~\cite{schwede2024singularities}.\\

Throughout Sections 2 to 6, $k$ is an algebraically closed field of characteristic $p>0$,
$R=k[[\underline{x}]]=k[[x_1,\dots,x_n]]$ is the formal power series ring,
and $\mathfrak m= \langle x_1,\dots,x_n \rangle$ is its maximal ideal. We first recall the definitions and properties of isolated singularities used throughout this section:

\begin{definition}\label{def:iso_sing}
Let $f,g\in\mathfrak{m}\subset\kx$.
\begin{enumerate}
    \item \textbf{The Jacobian ideal} of $f$ is defined to be $\Jac(f)=\langle\partial_1(f),\dots,\partial_n(f)\rangle$, the ideal generated by the formal partial derivatives of $f$.
    \item The \textbf{Milnor} and \textbf{Tjurina} numbers of $f$ are, respectively,
    \[
        \mu(f)=\dim_k\!\left(\frac{\kx}{\Jac(f)}\right),\qquad
        \tau(f)=\dim_k\!\left(\frac{\kx}{\Jac(f)+\langle f\rangle}\right).
    \]
    \item We say that $f$ and $g$ are \textbf{right} (resp.\ \textbf{contact}) equivalent and denote $f\equivr g$ (resp.\ $f\equivk g$) if there exists a $k$-automorphism $\varphi$ of $\kx$ such that $\varphi(f)=g$ (resp.\ $\langle\varphi(f)\rangle=\langle g\rangle$).
    \item Given an ideal $I\subset\kx$ and $0\neq f\in\kx$, we denote by $\ord_I(f)$ the largest $N$ such that $f\in I^N$, and we write $\ord_{\mathfrak{m}}(f)=\ord(f)$. We set $\ord_I(0)=\infty$.
    \item Given $w=(w_1,\dots,w_n)\in\mathbb{N}_{>0}^n$, we set $|w|=w_1+\cdots+w_n$, and for $\underline{x}^a=x_1^{a_1}\cdots x_n^{a_n}$ we define the \textbf{$w$-weight} $\operatorname{wt}_w(\underline{x}^a)=\sum_iw_ia_i$. If $0\neq f=\sum_ac_a\underline{x}^a$, the \textbf{$w$-order} of $f$ is $\ord_w(f)=\min\{\operatorname{wt}_w(\underline{x}^a):c_a\neq0\}$, and the \textbf{$w$-initial form} $\operatorname{in}_w(f)$ of $f$ is the sum of the terms of $f$ of $w$-weight $\ord_w(f)$. For $w=(1,\dots,1)$, we have $\ord_w(f)=\ord(f)$, and we write $f_{\ord(f)}=\operatorname{in}_w(f)$ for the lowest-degree form of $f$.
    \item A polynomial $g\in k[\underline{x}]$ is \textbf{$w$-homogeneous}, or \textbf{quasi-homogeneous}, of degree $D$ if every monomial of $g$ has $w$-weight $D$. We say that $f$ is \textbf{semi-quasi-homogeneous} if $\mu(\operatorname{in}_w(f))<\infty$ for some $w\in\mathbb{N}_{>0}^n$.
    \item\label{def:ADE} Over an algebraically closed field of characteristic $p>5$, the \emph{ADE plane-curve normal forms} are
$A_r:x^2+y^{r+1}$ ($r\geq1$), $D_r:x^2y+y^{r-1}$ ($r\geq4$), and
$E_6:x^3+y^4$, $E_7:x^3+xy^3$, $E_8:x^3+y^5$.
Here "of ADE type" means contact equivalent to one of these forms (see~\cite{greuel1990simple}).
\end{enumerate}
\end{definition}

In the following proposition we summarize some known results about isolated singularities that we use throughout this text:

\begin{proposition}\label{prop:mt_basic}
Let $f\in\mathfrak{m}\subset\kx$. Then:
\begin{enumerate}
    \item If $\tau(f)<\infty$, then $\mu(f)<\infty$ if and only if $f\in\sqrt{\Jac(f)}$. In addition, if $f^N\in\Jac(f)$, then $\tau(f)\leq\mu(f)\leq N\cdot\tau(f)$.
    \item (Detecting Smoothness) $\mu(f)=0$ if and only if $\tau(f)=0$ if and only if $f\equivr x_1$.
    \item (Stability under $\equivk$) Given $g\in\kx$, if $f\equivk g$, then $\tau(f)=\tau(g)$.
    \item (Stability under $\equivr$) Given $g\in\kx$, if $f\equivr g$, then $\mu(f)=\mu(g)$.
    \item (Contact Determinacy) $\tau(f)<\infty$ if and only if there exists $N>0$ such that for every $g\in\kx$ with $g-f\in\mathfrak{m}^{N+1}$ we have $f\equivk g$. In this case we can choose $N=2\tau(f)-\ord(f)+2$.
    \item (Right Determinacy) $\mu(f)<\infty$ if and only if there exists $N>0$ such that for every $g\in\kx$ with $g-f\in\mathfrak{m}^{N+1}$ we have $f\equivr g$. In this case we can choose $N=2\mu(f)-\ord(f)+2$.
    \item If $\mu(f)<\infty$ (resp.\ $\tau(f)<\infty$), then $\mathfrak{m}^{\mu(f)}\subset\Jac(f)$ (resp.\ $\mathfrak{m}^{\tau(f)}\subset\Jac(f)+\langle f\rangle$).
\end{enumerate}
\end{proposition}

\begin{proof}
For the first item, see Lemma 4.21 in~\cite{svoray2025detecting}, which in turn generalizes results from~\cite{liu2018milnor}, and Proposition 2.1 in~\cite{hefez2019hypersurface}. For the second item, see Lemma 2.44 in Chapter I of~\cite{greuel2007introduction}. For the third and fourth items, see Lemma 1.2.7 in~\cite{boubakri2009hypersurface}. For the fifth and sixth items, see Corollary 2.4 and Theorem 2.8 in~\cite{boubakri2012invariants}; for the stated right-determinacy bound, note that their sharper bound $2\mu(f)-d(f)+1$, in terms of the differential order $d(f)$, is at most $2\mu(f)-\ord(f)+2$. For the seventh item, see Proposition 3.1.17 in~\cite{boubakri2009hypersurface}.
\end{proof}

\begin{remark}
    \begin{enumerate}
        \item \textup{With respect to the first item of Proposition~\ref{prop:mt_basic}, the quotient $\frac{\mu(f)}{\tau(f)}$ has been studied and bounded for singular power series with complex coefficients and finite Milnor number. For example, Liu in~\cite{liu2018milnor} showed that this quotient is bounded by $n$, the number of variables. In the case $n=2$, the bound $\frac{4}{3}$ was proved in~\cite{almiron2019note} for curves with one Puiseux pair and for semi-quasi-homogeneous curves, and in general in~\cite{almiron2022quotient}; the latter reference also gives the bound $\frac{3}{2}$ for several families of surfaces with isolated singularities. Finally, Ma and Zuo in~\cite{ma2026quotient} presented the optimal bound for this quotient in arbitrary dimension, with an analogous bound for the quotient of the generalized Milnor number by the Tjurina number in positive characteristic. For more information, see~\cite{almiron2022quotient,alberich2021minimal}.}
    \item \textup{Saito in~\cite{saito1971quasihomogene} showed that for complex power series $\mu(f)-\tau(f)$ equals the $(n-1)$st Poincar\'e number, and therefore $\mu(f)=\tau(f)$ if and only if $f$ is right equivalent to a quasi-homogeneous polynomial. This is not true in the positive characteristic setting. For example, for $f=x^p+x^{p+1}+y^{p+1}$ we have that $\mu(f)=\tau(f)$ but $f$ is not right- equivalent to a quasi-homogeneus polynomial. Even if we replace right equivalence by contact equivalence, this is still not true (for example, take $f=(x^p+y^{p+1})(x^{2p}+y^{2p+1})$). }
    \end{enumerate}
\end{remark}

\begin{lemma}\label{lem:jac_frob}
Let $e\geq1$. If $f\in\mathfrak{m}^{[p^e]}$, then $\Jac(f)\subset\mathfrak{m}^{[p^e]}$.
\end{lemma}

\begin{proof}
If $f\in\mathfrak{m}^{[p^e]}$, then there exist $a_1,\dots,a_n\in\kx$ such that $f=\sum_{j=1}^n a_jx_j^{p^e}$. Therefore, for every $i$ we have $\partial_i(f)=\sum_{j=1}^n(\partial_i(a_j)x_j^{p^e}+a_j\partial_i(x_j^{p^e}))$. Since $k$ has characteristic $p$,  we have that $\partial_i(x_j^{p^e})=0$ for every $i,j$. Thus $\partial_i(f)=\sum_{j=1}^n\partial_i(a_j)x_j^{p^e}\in\mathfrak{m}^{[p^e]}$, and the result follows.
\end{proof}

We now define the $F$-pure threshold of a power series:

\begin{definition}\label{def:Fpt}
For an ideal $J\subset\kx$ and $q=p^e$, we denote by $J^{[q]}=\langle g^q:g\in J\rangle$ the \textbf{Frobenius power} of $J$. Given a proper ideal $J\subsetneq\kx$ and $f\in\sqrt J$, we denote by $\nu_e(f,J)$ the largest integer $N\geq0$ such that $f^N\notin J^{[p^e]}$, and we set $c(f,J)=\lim_{e\to\infty}\frac{\nu_e(f,J)}{p^e}$. If $J=\mathfrak{m}$, we write $\nu_e(f)=\nu_e(f,\mathfrak{m})$ and $\fpt(f)=c(f,\mathfrak{m})$, which we call the \textbf{$F$-pure threshold} of $f$.
\end{definition}

\begin{remark}\label{rem:fpt_contact}
\begin{enumerate}
    \item \textup{Observe that $\fpt(f)$ is stable under contact, and therefore also under right, equivalence. Indeed, a $k$-automorphism $\varphi$ of $\kx$ is local, so $\varphi(\mathfrak{m}^{[p^e]})=\mathfrak{m}^{[p^e]}$, and multiplication by a unit does not affect membership in $\mathfrak{m}^{[p^e]}$.}
    \item \textup{Assume that $p>2$ and let $f\in\kx$ be of order $2$. By the splitting lemma (see, for example, Theorem 2.1 in~\cite{greuel2025splitting}), we have $f\equivr x_1^2+\cdots+x_r^2+g(x_{r+1},\dots,x_n)$, where $r$ is the rank of the quadratic part and $\ord(g)\geq3$. If $r\geq2$, then by Fedder's criterion (see~\cite{fedder1983F} or Chapter 4.1 in~\cite{schwede2024singularities}) we have $\nu_e(f)=p^e-1$, and hence $\fpt(f)=1$.}
\end{enumerate}
\end{remark}

The following lemma is based on the proof of Lemma 3.3 in~\cite{schwede2024singularities} and the proof of Proposition 1.9 in~\cite{mustata2005f}:

\begin{lemma}\label{lem:bound_c_nu}
Given a proper ideal $I\subsetneq\kx$ and $f\in\sqrt I$, for every $e$ we have
\[
    p\nu_e(f,I)\leq\nu_{e+1}(f,I)\leq p(\nu_e(f,I)+1)-1.
\]
In particular, $\nu_0(f,I)\leq c(f,I)\leq\nu_0(f,I)+1$.
\end{lemma}

\begin{proof}
Since $\kx$ is a regular local ring, the Frobenius map is faithfully flat. Therefore $f^{\nu_e(f,I)}\notin I^{[p^e]}$ implies that $f^{p\nu_e(f,I)}\notin I^{[p^{e+1}]}$. Thus, $p\nu_e(f,I)\leq\nu_{e+1}(f,I)$, and by induction $p^e\nu_0(f,I)\leq\nu_e(f,I)$. Similarly, since $f^{\nu_e(f,I)+1}\in I^{[p^e]}$, we have $f^{p(\nu_e(f,I)+1)}\in I^{[p^{e+1}]}$. Hence $p(\nu_e(f,I)+1)\geq\nu_{e+1}(f,I)+1$. Thus $\frac{\nu_e(f,I)}{p^e}$ is non-decreasing and bounded above by $\nu_0(f,I)+1$, so its limit exists. Since $\nu_0(f,I)\leq\frac{\nu_e(f,I)}{p^e}\leq\nu_0(f,I)+1$, the result follows by taking $e\to\infty$.
\end{proof}

\begin{proposition}\label{prop:weighted_fpt}
Let $0\neq f\in\mathfrak{m}\subset k[[x_1,\dots,x_n]]$, let $w=(w_1,\dots,w_n)\in\mathbb{N}_{>0}^n$, set $d=\ord_w(f)$ and $g=\operatorname{in}_w(f)$. Then
\[
    \fpt(g)\leq\fpt(f)\leq
    \min\left\{1,\frac{w_1+\cdots+w_n}{d}\right\}.
\]
\end{proposition}

\begin{proof}
Write $f=g+h$, where every monomial of $h$ has $w$-weight greater than $d$, and fix $q=p^e$. If $g^N\notin\mathfrak{m}^{[q]}$, then some monomial $\underline{x}^a$ of $g^N$ has nonzero coefficient and $a_i<q$ for every $i$. It has $w$-weight $Nd$, while every monomial of $f^N-g^N$ has $w$-weight greater than $Nd$; hence it cannot cancel in $f^N$. Thus $\nu_e(f)\geq\nu_e(g)$ for every $e$, and $\fpt(f)\geq\fpt(g)$. Conversely, every monomial of $f^N$ has $w$-weight at least $Nd$. A monomial outside $\mathfrak{m}^{[q]}$ has exponents at most $q-1$, and therefore has $w$-weight at most $(q-1)(w_1+\cdots+w_n)$. Hence $Nd>(q-1)(w_1+\cdots+w_n)$ implies $f^N\in\mathfrak{m}^{[q]}$. Dividing the resulting bound on $\nu_e(f)$ by $q$ and letting $e\to\infty$ gives $\fpt(f)\leq \frac{w_1+\cdots+w_n}{d}$. Finally, $\fpt(f)\leq1$ since $f^q\in\mathfrak{m}^{[q]}$.
\end{proof}

We next record a Thom--Sebastiani type lower bound for the $F$-pure threshold of a sum of power series in disjoint sets of variables. We use the following notation for base-$p$ expansions.

\begin{definition}\label{def:basep}
For a real number $0<\alpha\leq1$, we write $\alpha=\sum_{l\geq1}\alpha_lp^{-l}$ for its base-$p$ expansion having infinitely many nonzero digits $\alpha_l\in\{0,\dots,p-1\}$, and we put $\langle\alpha\rangle_L=\sum_{l=1}^L\alpha_lp^{-l}$ for $L\geq0$. Since $0<\alpha-\langle\alpha\rangle_L\leq p^{-L}$, we have $p^L\langle\alpha\rangle_L=\lceil p^L\alpha\rceil-1$. We set
\[
    \theta_L(\alpha)=\lceil p^L\alpha\rceil-p^L\alpha\in[0,1),
    \qquad\text{so that}\qquad
    \alpha-\langle\alpha\rangle_L=\frac{1-\theta_L(\alpha)}{p^L}.
\]
For an integer $b\geq2$, we have $\theta_L(\frac1b)=0$ if and only if $b\mid p^L$; otherwise $b\lceil\frac{p^L}{b}\rceil-p^L$ is a positive integer, so $\theta_L(\frac1b)\geq\frac1b$. Given $0<\alpha,\beta\leq1$ with digits $\alpha_l,\beta_l$ as above, the \textbf{carry index} $L_p(\alpha,\beta)\in\mathbb{Z}_{\geq0}\cup\{\infty\}$ is the largest $L$ such that $\alpha_l+\beta_l\leq p-1$ for all $l\leq L$. We say that the base-$p$ digits of $\alpha$ and $\beta$ \textbf{add without carrying} if $L_p(\alpha,\beta)=\infty$.
\end{definition}

\begin{lemma}\label{lem:thom_sebastiani}
Let $0\neq g\in\langle x_1,\dots,x_u\rangle\subset k[[x_1,\dots,x_u]]$ and $0\neq h\in\langle y_1,\dots,y_v\rangle\subset k[[y_1,\dots,y_v]]$, and regard $g+h$ as an element of $k[[x_1,\dots,x_u,y_1,\dots,y_v]]$. Set $\alpha=\fpt(g)$, $\beta=\fpt(h)$ and $L=L_p(\alpha,\beta)$. If $L=\infty$, then $\fpt(g+h)\geq\alpha+\beta$. If $L<\infty$, then
\[
    \fpt(g+h)\geq\langle\alpha\rangle_L+\langle\beta\rangle_L+\frac{1}{p^L}
    =\alpha+\beta-\frac{1-\theta_L(\alpha)-\theta_L(\beta)}{p^L}.
\]
In particular, $\fpt(g+h)\geq\min\{1,\alpha+\beta-\frac1p\}$.
\end{lemma}

\begin{proof}
Write $\underline{x}=(x_1,\dots,x_u)$ and $\underline{y}=(y_1,\dots,y_v)$, fix $q=p^e$, and set $A=\frac{k[[\underline{x}]]}{\langle\underline{x}\rangle^{[q]}}$ and $B=\frac{k[[\underline{y}]]}{\langle\underline{y}\rangle^{[q]}}$, so that $\frac{k[[\underline{x},\underline{y}]]}{\mathfrak{m}^{[q]}}\cong A\otimes_kB$. By Lemma~\ref{lem:bound_c_nu}, $\frac{\nu_e(g)+1}{q}$ is non-increasing with limit $\alpha$, so $\nu_e(g)\geq\lceil q\alpha\rceil-1=q\langle\alpha\rangle_e$, and similarly $\nu_e(h)\geq q\langle\beta\rangle_e$. The images of $1,g,\dots,g^{\nu_e(g)}$ in $A$ are linearly independent: if $\sum_tc_tg^t=0$ in $A$ and $c_{t_0}$ is the first nonzero coefficient, then $g^{t_0}$ times a unit vanishes in $A$, which is impossible. Now let $i\leq\nu_e(g)$ and $j\leq\nu_e(h)$ add without carrying in base $p$. In $A\otimes_kB$, the terms of $(g+h)^{i+j}=\sum_t\binom{i+j}{t}g^t\otimes h^{i+j-t}$ with $t>\nu_e(g)$ vanish, and with respect to a basis of $A$ containing $1,g,\dots,g^{\nu_e(g)}$, the coefficient of $g^i$ is $\binom{i+j}{i}h^j\neq0$ by Lucas's theorem~\cite{lucas1878theorie}. Hence $\nu_e(g+h)\geq i+j$.

Let $\alpha_l,\beta_l$ be the digits of $\alpha,\beta$, so that $q\langle\alpha\rangle_e$ has base-$p$ digits $\alpha_1,\dots,\alpha_e$. If $L=\infty$, take $i=q\langle\alpha\rangle_e$ and $j=q\langle\beta\rangle_e$ and let $e\to\infty$. If $L<\infty$ and $e>L$, take $i=q\langle\alpha\rangle_e$ and let $j$ have base-$p$ digits $\beta_1,\dots,\beta_L,p-1-\alpha_{L+1},\dots,p-1-\alpha_e$. Then $i$ and $j$ add without carrying, and $j<q\langle\beta\rangle_e$ as $p-1-\alpha_{L+1}<\beta_{L+1}$. Since $i+j=q(\langle\alpha\rangle_L+\langle\beta\rangle_L)+p^{e-L}-1$, letting $e\to\infty$ gives the stated bound, by Definition~\ref{def:basep}. It equals $1$ if $L=0$, and it is at least $\alpha+\beta-\frac1p$ if $1\leq L<\infty$.
\end{proof}

Proposition~\ref{prop:weighted_fpt} also gives a lower bound in terms of the tangent cone. In two variables, this is the last assertion of Proposition~\ref{prop:MT_curves}.

\begin{proposition}\label{prop:tangent_cone}
Let $f\in\mathfrak{m}^2\subset k[[x_1,\dots,x_n]]$ with $n\geq2$, set $m=\ord(f)$, and assume that the lowest-degree form $f_m$ of $f$ is not of the form $u\ell^m$ with $u\in k^\times$ and $\ell$ a linear form. Then $\fpt(f)\geq\fpt(f_m)\geq\frac{1}{m-1}$.
\end{proposition}

\begin{proof}
The first inequality is Proposition~\ref{prop:weighted_fpt} with $w=(1,\dots,1)$. The zero set $Z\subset\mathbb{P}^{n-1}_k$ of $f_m$ is not contained in a hyperplane $\{\ell=0\}$, since otherwise $\ell^N\in\langle f_m\rangle$ for some $N$ by the Nullstellensatz, and unique factorization would give $f_m=u\ell^m$. Hence we may choose coordinates in which the coordinate points lie in $Z$. Then no monomial $x_i^m$ occurs in $f_m$, so every monomial of $f_m$ has all exponents at most $m-1$. A vertex $v$ of the Newton polytope of $f_m$ is the unique minimizer on it of a linear function with rational coefficients; adding a multiple of $(1,\dots,1)$ and clearing denominators gives $w'\in\mathbb{N}_{>0}^n$ with $\operatorname{in}_{w'}(f_m)=\kappa\underline{x}^v$ for some $\kappa\in k^\times$. As $(\underline{x}^v)^N\notin\mathfrak{m}^{[q]}$ if and only if $N\max_iv_i\leq q-1$, Proposition~\ref{prop:weighted_fpt} gives $\fpt(f_m)\geq\fpt(\underline{x}^v)=\frac{1}{\max_iv_i}\geq\frac{1}{m-1}$.
\end{proof}

We next relate the $F$-pure threshold to the Milnor and Tjurina numbers.

\begin{lemma}\label{lem:c(f,Jac)_fpt}
Let $0\neq f\in\mathfrak{m}$. Then:
\begin{enumerate}
    \item\label{thm:fpt_milnor}  If $0<\mu(f)<\infty$, then $\fpt(f)\mu(f)\geq c(f,\Jac(f))$. In particular we have that $\fpt(f)\geq\frac{1}{\tau(f)}-\frac{1}{\mu(f)}$.
    \item If $0<\tau(f)<\infty$, then $\fpt(f)\tau(f)\geq c(f,\Jac(f)+\langle f\rangle)$.
\end{enumerate}
\end{lemma}

\begin{proof}
If $0<\mu(f)<\infty$, Proposition~\ref{prop:mt_basic} gives $\mathfrak{m}^{\mu(f)}\subset\Jac(f)$, and hence $(\mathfrak{m}^{[p^e]})^{\mu(f)}\subset\Jac(f)^{[p^e]}$. Since $f^{\nu_e(f)+1}\in\mathfrak{m}^{[p^e]}$, it follows that $\mu(f)(\nu_e(f)+1)\geq\nu_e(f,\Jac(f))+1$. Dividing by $p^e$ and taking the limit proves the first assertion, with the second part following from Lemma~\ref{lem:bound_c_nu}, together with item 1 of Proposition~\ref{prop:mt_basic}, which gives us that $\frac{\mu(f)}{\tau(f)}\leq\nu_0(f,\Jac(f))+1$.  Similarly, if $0<\tau(f)<\infty$, then $\mathfrak{m}^{\tau(f)}\subset\Jac(f)+\langle f\rangle$, so $(\mathfrak{m}^{[p^e]})^{\tau(f)}
\subset(\Jac(f)+\langle f\rangle)^{[p^e]}$.  The same argument gives $\tau(f)(\nu_e(f)+1)
\geq\nu_e(f,\Jac(f)+\langle f\rangle)+1$, and the second assertion follows after dividing by $p^e$ and taking the limit.
\end{proof}

\begin{corollary}\label{cor:BS_F}
If $0<\mu(f)<\infty$, then $\fpt(f)\geq\frac{\nu_0(f,\Jac(f))}{\mu(f)}$.
\end{corollary}

\begin{proof}
This follows from Lemmas~\ref{lem:bound_c_nu} and~\ref{lem:c(f,Jac)_fpt}.
\end{proof}

\begin{corollary}\label{cor:tau_ord}
If $0<\tau(f)<\infty$, then $\fpt(f)\geq\frac{1}{\tau(f)}\left(1-\frac{1}{\ord(f)}\right)$.
\end{corollary}

\begin{proof}
Note that $f\notin\mathfrak{m}^{\ord(f)+1}$ and $\Jac(f)\subset\mathfrak{m}^{\ord(f)-1}$. Therefore, from item 3 of Proposition 1.7 in~\cite{mustata2005f} we can conclude that $c(f,\Jac(f)+\langle f\rangle)$ is bounded below by $1-\frac{1}{\ord(f)}$, and the result follows from Lemma~\ref{lem:c(f,Jac)_fpt}.
\end{proof}

Item~\ref{thm:fpt_milnor} of Lemma~\ref{lem:c(f,Jac)_fpt} leads directly to the question of what happens if $\tau(f)$ is finite but $\mu(f)$ is not. In this case, regarding $\frac{1}{\mu(f)}$ as $0$, one would expect $\fpt(f)\geq\frac{1}{\tau(f)}$. This is, in fact, true and appears as Proposition~\ref{prop:milnor_tjurina}. It follows from the following lemma, whose proof is inspired by the proof of Proposition 3.3 in~\cite{pellikaan1989series}:

\begin{lemma}\label{lem:milnor_limit}
Let $f\in\mathfrak m\subset\kx$ such that $\mu(f)=\infty$ and $\tau(f)<\infty$. Then there exist $a\in\mathfrak{m}\setminus\mathfrak{m}^2$ and $l_0\gg0$ such that for every $l>l_0$ not divisible by $p$ we have $\mu(f+a^l)<\infty$.
\end{lemma}

\begin{proof}
Since $\tau(f)<\infty$, the ideal $\langle f\rangle+\Jac(f)$ is $\mathfrak{m}$-primary, while $\mu(f)=\infty$ implies that $\frac{\kx}{\Jac(f)}$ is one-dimensional. Since $k$ is infinite and $\operatorname{ht}(\Jac(f))=n-1$, general $n-1$ linear combinations of its generators form a regular sequence (see~\cite{eisenbud1995commutative}; compare~\cite{ulrich1987theory,huneke1987structure}). By the chain rule, the derivatives $\partial_2(f),\dots,\partial_n(f)$ after a general linear coordinate change are such combinations. We may simultaneously choose $x_1$ outside the finitely many minimal primes of $\Jac(f)$: both requirements hold on nonempty open sets of linear coordinate changes. Thus $\mathfrak a=\langle\partial_2(f),\dots,\partial_n(f)\rangle$ has height $n-1$ and $\Jac(f)+\langle x_1\rangle$ is $\mathfrak m$-primary. The quotient $\frac{\kx}{\mathfrak a}$ is a one-dimensional complete intersection, so its associated primes $\mathfrak p_1,\dots,\mathfrak p_r$ are all minimal. If $x_1\in\mathfrak p_i$, then $\partial_1(f)\notin\mathfrak p_i$, since otherwise $\Jac(f)+\langle x_1\rangle\subset\mathfrak p_i$. Hence $\partial_1(f)+lx_1^{l-1}\notin\mathfrak p_i$ for every $l\geq2$. If $x_1\notin\mathfrak p_i$, there is at most one $l\geq2$, with $p\nmid l$, such that $\partial_1(f)+lx_1^{l-1}\in\mathfrak p_i$. Indeed, if this holds for $l_1>l_2$, then $x_1^{l_2-1}(l_1x_1^{l_1-l_2}-l_2)\in\mathfrak{p}_i$. Since $x_1\notin\mathfrak{p}_i$ and $l_2$ is nonzero in $k$, the second factor is a unit, a contradiction. Thus, for every sufficiently large $l$ not divisible by $p$, the element $\partial_1(f)+lx_1^{l-1}$ is a nonzerodivisor on $\frac{\kx}{\mathfrak{a}}$. Consequently $\Jac(f+x_1^l)=\mathfrak{a}+\langle\partial_1(f)+lx_1^{l-1}\rangle$ is $\mathfrak{m}$-primary, and hence $\mu(f+x_1^l)<\infty$.
\end{proof}

\begin{remark}\label{rem:limit_milnor}
\textup{In Lemma~\ref{lem:milnor_limit}, $\mu(f+a^l)\to\infty$ as $l\to\infty$
through integers not divisible by $p$. Otherwise there are an integer $M$ and
an unbounded sequence $l_j$ with $p\nmid l_j$ such that
$\mu(f+a^{l_j})\leq M$. For $j\gg0$, we have
$\ord(f+a^{l_j})=\ord(f)=d$, so the right-determinacy theorem
(Proposition~\ref{prop:mt_basic}) gives us that each  $f+a^{l_j}$ is
$N$-right-determined for the common bound $N=2M-d+2$.
Since $a^{l_j}\in\mathfrak m^{N+1}$ for $j\gg0$, we obtain
$f+a^{l_j}\equivr f$, forcing $\mu(f)<\infty$, a contradiction.}
\end{remark}

\begin{proposition}\label{prop:milnor_tjurina}
Let $f\in\mathfrak{m}^2$.
 If $\tau(f)<\infty$ and $\mu(f)=\infty$, then $\fpt(f)\geq\frac{1}{\tau(f)}$.
\end{proposition}

\begin{proof}
By Lemma~\ref{lem:milnor_limit}, there exists $a$ such that $f_l=f+a^l$ satisfies $\mu(f_l)<\infty$ for all sufficiently large $l$ not divisible by $p$. By contact determinacy, $f_l\equivk f$ for all sufficiently large such $l$, because $a^l$ has arbitrarily high order. Hence $\tau(f_l)=\tau(f)$ and, by Remark~\ref{rem:fpt_contact}, $\fpt(f_l)=\fpt(f)$. Item~\ref{thm:fpt_milnor} of Lemma~\ref{lem:c(f,Jac)_fpt} therefore gives
\[
    \fpt(f)\geq\frac{1}{\tau(f)}-\frac{1}{\mu(f_l)}.
\]
By Remark~\ref{rem:limit_milnor}, $\mu(f_l)\to\infty$ as $l\to\infty$ through integers not divisible by $p$. Taking the limit yields $\fpt(f)\geq\frac{1}{\tau(f)}$.
\end{proof}

\begin{remark}\label{rem:ord_comparison}
\textup{Unfortunately, the bounds of Lemma~\ref{lem:c(f,Jac)_fpt} and of Proposition~\ref{prop:milnor_tjurina} never improve on the classical bound $\fpt(f)\geq\frac{1}{\ord(f)}$ (see Proposition 4.1 in~\cite{takagi2004f} and Remark~\ref{rem:2dim_sing}). Indeed, let $d=\ord(f)$. Since $\Jac(f)+\langle f\rangle\subset\mathfrak{m}^{d-1}$, we have $\tau(f)\geq\binom{n+d-2}{n}$, with equality only if $\Jac(f)+\langle f\rangle=\mathfrak{m}^{d-1}$.} \textup{If $n\geq2$ and $d\geq3$, equality is impossible. As $f\in\mathfrak{m}^d$, it would require the degree $d-1$ parts of the $n$ partial derivatives to span the $\binom{n+d-2}{n-1}>n$ forms of degree $d-1$. Hence $\tau(f)>\binom{n+d-2}{n}\geq\binom d2\geq d$, and both bounds are smaller than $\frac1d$.} \textup{If $d=2$, then either $\tau(f)\geq2$, or $\Jac(f)+\mathfrak{m}^2=\mathfrak{m}$. In the latter case, $\Jac(f)=\mathfrak{m}$ by Nakayama's lemma, $\mu(f)=1$, and item~\ref{thm:fpt_milnor} gives $0$.} \textup{If $n=1$, then $f=ux_1^d$ with $u$ a unit, and either $\tau(f)\geq d$, or $\Jac(f)+\langle f\rangle=\langle x_1^{d-1}\rangle$. In the latter case, $\Jac(f)=\langle x_1^{d-1}\rangle$ by Nakayama's lemma, so $\mu(f)=\tau(f)$, and item~\ref{thm:fpt_milnor} gives $0$.} \textup{The same argument shows that Part~\ref{cor:lct_milnor} of Corollary~\ref{corB} never improves on $\lct(f)\geq\frac{1}{\ord(f)}$.}
\end{remark}

The following lemma is the classical $ADE$ classification of simple contact singularities over positive characteristic bigger than $5$. We present it and sketch its proof for the sake of completeness. For more details, see~\cite{greuel1990simple}.

\begin{lemma}[ADE reduction in small multiplicity]\label{lem:ADE_reduction}
Assume $p>5$ and $f\in\mathfrak m^2\subset k[[x,y]]$ has $\tau(f)<\infty$.
Then at least one of the following holds:
\begin{enumerate}
    \item $\ord(f)\geq4$. 
    \item After contact equivalence, $f\in \langle x,y^2 \rangle^3$. 
    \item  $f$ is of ADE type.
\end{enumerate}
\end{lemma}
\begin{proof}
If $\ord(f)\geq4$, there is nothing to prove. If $\ord(f)=2$, Weierstrass preparation and completion of the square, followed by multiplication by a unit and a rescaling of $x$, give $f\equivk x^2+y^s$ for some $s\geq2$. The exponent $s$ is finite because $\tau(f)<\infty$.

Suppose that $\ord(f)=3$ and the tangent cone has at least two distinct linear factors. One of them is simple, so Hensel's lemma (Lemma~\ref{lem:hensel}) gives $f=f_1f_2$ with $f_1$ smooth and $\ord(f_2)=2$, and their tangent cones coprime. Choose coordinates with $f_1=y$. Weierstrass preparation in $x$ and completion of the square give $f\equivk y(x^2+b(y))$. Here $\ord(b)\geq2$, and $b\neq0$ since otherwise the curve would be non-reduced. Rescaling $x$ by a square root of the unit part of $b$ gives $f\equivk x^2y+y^{s+1}$ for some $s\geq2$, a $D_{s+2}$ form.

It remains to consider a triple tangent. After a linear change of coordinates, Weierstrass preparation and elimination of the $x^2$ term give
\[
    f\equivk x^3+b(y)x+c(y),\qquad \ord(b)\geq3,\quad\ord(c)\geq4.
\]
If $\ord(c)=4$, or if $\ord(c)\geq5$ and $\ord(b)=3$, or if $\ord(c)=5$ and $\ord(b)\geq4$, the weighted initial form is, after rescaling, respectively $x^3+y^4$, $x^3+xy^3$, or $x^3+y^5$. These have weights $(4,3)$, $(3,2)$, $(5,3)$ and degrees $12,9,15$. Their Jacobian algebras have highest weighted degrees $10,8,14$, respectively, so every term of higher weight than the corresponding defining equation lies in its Jacobian ideal. Solving successively in each weighted degree removes these terms by coordinate changes of strictly higher weight; the resulting changes converge because the weights are positive. This gives the forms $E_6,E_7,E_8$, in agreement with~\cite{greuel1990simple,greuel2017survey}. In the remaining case, $\ord(b)\geq4$ and $\ord(c)\geq6$, so the equation belongs to $\langle x,y^2\rangle^3$.
\end{proof}

\begin{example}\label{ex:ADE}
\textup{Under the hypotheses of Lemma~\ref{lem:ADE_reduction}, if $f$ is not
of ADE type then $\fpt(f)\leq\frac12$. Indeed, in case (i) the ordinary
weighted-degree estimate gives $\fpt(f)\leq2/\ord(f)\leq\frac12$;
in case (ii), weights $(2,1)$ give $\ord_w(f)\geq6$, hence
$\fpt(f)\leq3/6=\frac12$ by Proposition~\ref{prop:weighted_fpt}.
For ADE forms, the standard computations give $\fpt(f)>\frac12$;
see Example 2.5 in~\cite{takagi2004f} and the explicit forms of item~\ref{def:ADE} in 
Definition~\ref{def:iso_sing}. In particular the same characterization holds
for every $f\in\mathfrak m^2$ with finite Tjurina number.}
\end{example}

\begin{proposition}\label{prop:det_fpt_m-adic}
Let $0\neq f\in\mathfrak{m}^2$. Then:
\begin{enumerate}
    \item\label{det:constancy} There exists $N\gg0$ such that for every $g\in(\Jac(f)+\langle f\rangle)^N$ we have $\fpt(f)=\fpt(f+g)$.
    \item\label{det:continuity} For every $0\neq g\in\mathfrak{m}$ with $f+g\neq0$, we have $|\fpt(f+g)-\fpt(f)|\leq\fpt(g)\leq\frac{n}{\ord(g)}$. In particular, $\fpt$ is continuous in the $\mathfrak{m}$-adic topology.
\end{enumerate}
\end{proposition}

\begin{proof}
For the first item, by Theorem 2.4 and Corollary 2.22 in~\cite{svoray2025detecting}, there exists $N$ such that for every $g\in(\Jac(f)+\langle f\rangle)^N$ we have $f\equivk f+g$. Therefore, by Remark~\ref{rem:fpt_contact}, $\fpt(f)=\fpt(f+g)$. For the second item, fix $q=p^e$. If $a\geq\nu_e(f)+\nu_e(g)+1$, then every term of $(f+g)^a=\sum_i\binom ai f^ig^{a-i}$ lies in $\mathfrak{m}^{[q]}$, as $i\geq\nu_e(f)+1$ or $a-i\geq\nu_e(g)+1$. Hence $\nu_e(f+g)\leq\nu_e(f)+\nu_e(g)$, and dividing by $q$ and letting $e\to\infty$ gives $\fpt(f+g)\leq\fpt(f)+\fpt(g)$. Applying this also to $f=(f+g)+(-g)$, and using $\fpt(-g)=\fpt(g)$, gives the first inequality, and Proposition~\ref{prop:weighted_fpt} with $w=(1,\dots,1)$ gives $\fpt(g)\leq\frac{n}{\ord(g)}$.
\end{proof}

\begin{remark}
\begin{enumerate}
    \item \textup{Item~\ref{det:constancy} of Proposition~\ref{prop:det_fpt_m-adic} is a generalization of Theorem 4.3 in~\cite{hernandez2018local}, where it is assumed that $f$ has an isolated singularity and hence $\sqrt{\langle f\rangle+\Jac(f)}=\mathfrak{m}$. In addition, by Proposition~\ref{prop:mt_basic}, if $\tau(f)<\infty$, then contact determinacy gives the explicit safe choice $N=2\tau(f)-\ord(f)+3$: indeed, $\Jac(f)+\langle f\rangle\subset\mathfrak{m}$, so its $N$th power is contained in $\mathfrak{m}^{N}$, which is deep enough for $(N-1)$-determinacy. This still improves the bound given in~\cite{hernandez2018local}, namely $N=p^{2\tau(f)}\cdot n$.\footnote{We note that the author of~\cite{hernandez2018local}, Hern\'{a}ndez, was informed of the potential application of determinacy to $\mathfrak{m}$-adic constancy of $\fpt$ by Holger Brenner in a private conversation.}}
    \item \textup{Not every $f\in\mathfrak{m}$ has this $\mathfrak{m}$-adic constancy: the version of item~\ref{det:constancy} of Proposition~\ref{prop:det_fpt_m-adic} obtained by replacing $\Jac(f)+\langle f\rangle$ with $\mathfrak{m}$ fails if $f$ does not define an isolated singularity. For example, over $k[[x,y]]$, if $d\geq2$ is coprime to $p$, then item~\ref{binom:old} of Lemma~\ref{lem:binomial} gives $\fpt(x^d+y^N)\geq\frac1d+\frac{1}{dN}>\fpt(x^d)=\frac1d$ for every $N\geq2$, although $y^N\in\mathfrak{m}^N$. Thus $\fpt(x^d+g)$ is not constant for $g$ in any power of $\mathfrak{m}$, even though it converges to $\fpt(x^d)$ by item~\ref{det:continuity} of Proposition~\ref{prop:det_fpt_m-adic}.}
    \item \textup{Lower semicontinuity in the $\mathfrak{m}$-adic topology also follows from item~\ref{det:continuity} of Proposition~\ref{prop:det_fpt_m-adic} and the ascending chain condition for $F$-pure thresholds, first stated as Conjecture 4.4 in~\cite{blickle2009F} and proved by Sato in~\cite{sato2019ascending}. Indeed, the latter gives $\varepsilon>0$ such that no $F$-pure threshold lies in $(\fpt(f)-\varepsilon,\fpt(f))$, so $\fpt(f+g)\geq\fpt(f)$ for every $g\in\mathfrak{m}^N$ with $N>\frac n\varepsilon$.}
\end{enumerate}
\end{remark}

To improve on Proposition~\ref{prop:milnor_tjurina}, we replace the containment $\mathfrak{m}^{\mu(f)}\subset\Jac(f)$ used in Lemma~\ref{lem:c(f,Jac)_fpt} by the socle degree of the Milnor algebra of a weighted initial form. We use the notation of Definition~\ref{def:iso_sing}. The next proposition is a consequence of the weighted Jacobian estimate of M\"uller, as presented in Lemma 3.5 in~\cite{muller2018f}, also obtained in Proposition 5.6 in~\cite{hernandez2016f}. We include a proof to make the dependence on the weights and on $p$ explicit.

\begin{proposition}\label{prop:qh_bound}
Let $w\in\mathbb{N}_{>0}^n$ and let $g\in k[\underline{x}]$ be $w$-homogeneous of degree $D$ with $0<\mu(g)<\infty$. Then
\[
    \fpt(g)\geq\min\left\{1,\frac{|w|}{D}-\frac{\max\{0,(n-1)D-|w|\}}{pD}\right\}.
\]
\end{proposition}

\begin{proof}
We grade the ring $S=k[\underline{x}]$ by setting $\deg(x_i)=w_i$. Let $I=\langle\partial_1(g),\dots,\partial_n(g)\rangle\subset S$. The associated primes of $\frac SI$ are homogeneous, so $\frac SI$ embeds into $\frac{\kx}{\Jac(g)}$, which is finite dimensional. Hence $I$ is primary to $\langle\underline{x}\rangle$, every $\partial_i(g)$ is nonzero by Krull's height theorem, and none is a nonzero constant since $\mu(g)>0$. Thus $\partial_1(g),\dots,\partial_n(g)$ is a homogeneous system of parameters, hence a regular sequence, in the Cohen--Macaulay ring $S$, of degrees $D-w_i>0$. By the Hilbert-series formula for graded complete intersections~\cite{eisenbud1995commutative}, the Hilbert series of $\frac SI$ is $\prod_i\frac{1-t^{D-w_i}}{1-t^{w_i}}$, a polynomial of degree $nD-2|w|$, and every monomial of $w$-weight greater than $nD-2|w|$ lies in $\Jac(g)$.

By Lemma~\ref{lem:bound_c_nu}, $\nu_{e+1}(g)+1\leq p(\nu_e(g)+1)$, so $\frac{\nu_e(g)+1}{p^e}$ is non-increasing with limit $\fpt(g)$. It suffices to bound it below for every $e$. As $g^{p^e}\in\mathfrak{m}^{[p^e]}$, we have $N=\nu_e(g)+1\leq p^e$, and we may assume $e\geq1$ and $N<p^e$. Let $p^t$ be the largest power of $p$ dividing $N$. Then $t<e$, and applying Lemma~\ref{lem:bound_c_nu} $t$ times gives $p^t\nu_{e-t}(g)\leq N-1$ and $N\leq p^t(\nu_{e-t}(g)+1)$. Hence $N'=\frac{N}{p^t}=\nu_{e-t}(g)+1$ and $p\nmid N'$. Set $q=p^{e-t}\geq p$.

By Lemma~\ref{lem:jac_frob}, we have $N'g^{N'-1}\partial_i(g)=\partial_i(g^{N'})\in\mathfrak{m}^{[q]}$, and as $N'$ is a unit, $g^{N'-1}\Jac(g)\subset\mathfrak{m}^{[q]}$. Since $g^{N'-1}\notin\mathfrak{m}^{[q]}$ and $\mathfrak{m}^{[q]}$ is a monomial ideal, $g^{N'-1}$ has a monomial $\underline{x}^a$ with all $a_i\leq q-1$, of $w$-weight $(N'-1)D$. For $b=(q-1,\dots,q-1)-a$, the monomial $(x_1\cdots x_n)^{q-1}$ appears in $\underline{x}^bg^{N'-1}$, since multiplication by $\underline{x}^b$ is injective on monomials. Hence $\underline{x}^bg^{N'-1}\notin\mathfrak{m}^{[q]}$, so $\underline{x}^b\notin\Jac(g)$, and therefore $(q-1)|w|-(N'-1)D=\operatorname{wt}_w(\underline{x}^b)\leq nD-2|w|$.

Rearranging gives
\[
    \frac{N}{p^e}=\frac{N'}{q}\geq\frac{|w|}{D}-\frac{(n-1)D-|w|}{qD}\geq\frac{|w|}{D}-\frac{\max\{0,(n-1)D-|w|\}}{pD},
\]
where the last inequality holds since $q\geq p$.
\end{proof}

As a consequence, we obtain a lower bound for semi-quasi-homogeneous power series:

\begin{proposition}\label{prop:sqh_bound}
Let $f\in\mathfrak{m}^2$ and $w\in\mathbb{N}_{>0}^n$, and set $D=\ord_w(f)$ and $g=\operatorname{in}_w(f)$. If $\mu(g)<\infty$, then
\[
    \fpt(f)\geq\min\left\{1,\frac{|w|}{D}-\frac{\max\{0,(n-1)D-|w|\}}{pD}\right\}.
\]
\end{proposition}

\begin{proof}
Since $f\in\mathfrak{m}^2$, every monomial of $g=\operatorname{in}_w(f)$ has degree at least $2$, so $\Jac(g)\subset\mathfrak{m}$ and $\mu(g)>0$. As $g$ is $w$-homogeneous of degree $\ord_w(f)$, the result follows from Propositions~\ref{prop:weighted_fpt} and~\ref{prop:qh_bound}.
\end{proof}

The hypothesis of Proposition~\ref{prop:sqh_milnor} below can be read off from $\mu(f)$. For $w=(1,\dots,1)$, the next lemma says that $\mu(f)\geq(\ord(f)-1)^n$, with equality if and only if $\mu(f_{\ord(f)})<\infty$.

\begin{lemma}\label{lem:milnor_min}
Let $0\neq f\in\mathfrak{m}$ with $\mu(f)<\infty$, and let $w\in\mathbb{N}_{>0}^n$ with $D=\ord_w(f)>\max_iw_i$. Then $\mu(f)\geq\prod_{i=1}^n\frac{D-w_i}{w_i}$, with equality if and only if $\mu(\operatorname{in}_w(f))<\infty$. More precisely, $\mu(f)\geq\prod_{i=1}^n\frac{\ord_w(\partial_i(f))}{w_i}$.
\end{lemma}

\begin{proof}
Let $R'=k[[t_1,\dots,t_n]]$, with maximal ideal $\mathfrak{n}$, be a $\kx$-algebra via $x_i\mapsto t_i^{w_i}$. It is free over $\kx$ with basis $\{\underline{t}^b:0\leq b_i<w_i\}$, and both rings have residue field $k$, so the length of $\frac{R'}{\Jac(f)R'}$ equals $\mu(f)\prod_iw_i$. A monomial of $w$-weight $s$ maps to one of degree $s$. So the image $G_i$ of $\partial_i(f)$, which is nonzero as $\mu(f)<\infty$, has order $\ord_w(\partial_i(f))\geq a_i=D-w_i\geq1$, with degree-$a_i$ part $P_i=\partial_i(\operatorname{in}_w(f))(t_1^{w_1},\dots,t_n^{w_n})$.

If $\mathfrak{n}^M\subset\Jac(f)R'$ with $M>\max_i\ord(G_i)$, then truncating each $G_i$ in degrees at most $M$ changes it by an element of $\mathfrak{n}\Jac(f)R'$. By Nakayama's lemma, we may therefore assume that the $G_i$ are polynomials. The divisors $V(G_i)\subset\mathbb{A}^n_k$ meet properly at the origin, with intersection multiplicity equal to the length of $\frac{R'}{\Jac(f)R'}$ and multiplicities $\ord(G_i)\geq a_i$.

By the local intersection-multiplicity inequality and its equality criterion (see Section 12 in~\cite{fulton1998intersection} and~\cite{serre2000local}), this intersection multiplicity is at least $\prod_i\ord(G_i)=\prod_i\ord_w(\partial_i(f))$, which is at least $\prod_ia_i$, and it equals $\prod_ia_i$ if and only if $\ord(G_i)=a_i$ for all $i$ and the tangent cones $V(P_i)$ meet only at the origin. This gives the inequalities.

If some $\ord(G_i)>a_i$, then $P_i=0$, so $\partial_i(\operatorname{in}_w(f))=0$ and $\mu(\operatorname{in}_w(f))=\infty$. Otherwise, as $k$ is algebraically closed, the $P_i$ have a common zero in $k^n\setminus\{0\}$ if and only if the $\partial_i(\operatorname{in}_w(f))$ do. By the Nullstellensatz~\cite{atiyah1969introduction} and the first paragraph of the proof of Proposition~\ref{prop:qh_bound}, the latter have none if and only if $\mu(\operatorname{in}_w(f))<\infty$.
\end{proof}

Combining Proposition~\ref{prop:sqh_bound} with Lemma~\ref{lem:milnor_min}, we obtain a lower bound that depends only on $n$, $p$ and $\mu(f)$:

\begin{proposition}\label{prop:sqh_milnor}
Let $n\geq2$ and $f\in\mathfrak{m}^2$, and assume that $\mu(f)=\prod_{i=1}^n\frac{D-w_i}{w_i}$ for some $w\in\mathbb{N}_{>0}^n$ with $D=\ord_w(f)>\max_iw_i$. Then $\fpt(f)\geq\lambda-\frac{n-1-\lambda}{p}$, where $\lambda=\min\{1,\frac{n}{1+\mu(f)^{1/n}}\}$. In particular, if $d=\ord(f)$ and $\mu(f)=(d-1)^n$, then $\fpt(f)\geq\min\left\{1,\frac{n}{d}-\frac{(n-1)d-n}{pd}\right\}$, which is larger than $\frac{1}{\ord(f)}$ whenever $p\geq d-1$.
\end{proposition}

\begin{proof}
By Lemma~\ref{lem:milnor_min}, $\mu(\operatorname{in}_w(f))<\infty$. Set $S=\frac{|w|}{D}$. Proposition~\ref{prop:sqh_bound} gives $\fpt(f)\geq\varphi(S)$, where $\varphi(s)=\min\{1,s-\frac{\max\{0,n-1-s\}}{p}\}$ is non-decreasing and $\varphi(\lambda)=\lambda-\frac{n-1-\lambda}{p}$, as $\lambda\leq1\leq n-1$. So it suffices to show $\lambda\leq S$, which is clear if $S\geq1$.

Let $S<1$ and $r_i=\frac{w_i}{D}$, so $\sum_ir_i=S$ and $\mu(f)=\prod_i(\frac1{r_i}-1)$. On $\{r\in(0,1)^n:\sum_ir_i=S\}$, every $r_i\leq S<1$. So $\sum_i\log(\frac1{r_i}-1)$ is bounded below by $n\log(\frac1S-1)$ and tends to $\infty$ as some $r_i\to0$. Its minimum is therefore attained at a critical point, where $r_i(1-r_i)$ is independent of $i$, so each $r_i\in\{\rho,1-\rho\}$. Two distinct values would give $S\geq1$, so all $r_i=\frac Sn$. Hence $\mu(f)\geq(\frac nS-1)^n$, that is, $S\geq\frac{n}{1+\mu(f)^{1/n}}\geq\lambda$.

Finally, if $\mu(f)=(d-1)^n$, Lemma~\ref{lem:milnor_min} with $w=(1,\dots,1)$ gives $\mu(f_d)<\infty$. So Proposition~\ref{prop:sqh_bound} applies with $w=(1,\dots,1)$, where $(n-1)d-n\geq0$. For $p\geq d-1$, we have $p(n-1)\geq(n-1)d-(n-1)>(n-1)d-n$, which is equivalent to $\frac nd-\frac{(n-1)d-n}{pd}>\frac1d$.
\end{proof}

\begin{example}\label{ex:sharp}
\textup{The bound of Proposition~\ref{prop:sqh_milnor} is attained, both when its hypothesis holds and when it fails.}
\begin{enumerate}
    \item \textup{Let $g=x^3+y^3$ and $p\equiv2\pmod3$. Then $\mu(g)=4=(\ord(g)-1)^2$, so Proposition~\ref{prop:sqh_milnor} gives $\fpt(g)\geq\frac23-\frac1{3p}$. Conversely, the monomials of $g^N$ are the $x^{3i}y^{3j}$ with $i+j=N$, with coefficient $\binom Ni$, and such a monomial lies outside $\langle x^p,y^p\rangle$ if and only if $i,j\leq\frac{p-2}{3}$. For $N=\frac{2(p-2)}{3}<p$, the monomial $x^{p-2}y^{p-2}$ has coefficient $\binom{N}{N/2}\not\equiv0\pmod p$. So $\nu_1(g)=\frac{2(p-2)}{3}$, and since $\frac{\nu_e(g)+1}{p^e}$ is non-increasing by Lemma~\ref{lem:bound_c_nu}, $\fpt(g)\leq\frac{\nu_1(g)+1}{p}=\frac23-\frac1{3p}$.}
    \item \textup{Let $2\leq n\leq p+1$ and $f=\sum_{i=1}^n(x_i^p+x_i^{p+1})$. Then $\Jac(f)=\mathfrak{m}^{[p]}$, so $\mu(f)=p^n$, $\lambda=\frac{n}{p+1}$, and the bound of Proposition~\ref{prop:sqh_milnor} is $\frac{\lambda(p+1)-(n-1)}{p}=\frac1p$. Since $f\in\mathfrak{m}^{[p]}$, we have $\fpt(f)\leq\frac{\nu_1(f)+1}{p}=\frac1p$, and $\fpt(f)\geq\frac{1}{\ord(f)}=\frac1p$. However, every $\operatorname{in}_w(f)$ is a sum of $p$-th powers, so by Lemma~\ref{lem:milnor_min} the hypothesis of Proposition~\ref{prop:sqh_milnor} fails for $f$.}
\end{enumerate}
\end{example}

Conjecture~\ref{conj:milnor} in the introduction asserts that the bound of Proposition~\ref{prop:sqh_milnor} holds without its hypothesis on $\mu(f)$.

\begin{remark}\label{rem:conj_consequences}
\textup{By Propositions~\ref{prop:lct_fpt} and~\ref{prop:tau_p}, Conjecture~\ref{conj:milnor} implies Part~\ref{cor:lct_milnor_min} of Corollary~\ref{corB} without its hypothesis on $\mu(f)$. Example~\ref{ex:sharp} gives families attaining the bound, both with and without the weighted hypothesis. The assumption $n\geq2$ is needed: for $n=1$ and $p\nmid d$, the function $f=x^d$ has $\lambda=\frac1d=\fpt(f)<\lambda+\frac{\lambda}{p}$. Since $\fpt(f)\geq\frac{1}{\ord(f)}$, the conjecture only has content when its bound exceeds $\frac{1}{\ord(f)}$. Theorem~\ref{thm:conj_cases} establishes the conjecture in the semi-quasi-homogeneous case, for $n=2$, and for $n\geq p+1$. In particular, for $n=2$, Part~\ref{cor:lct_milnor_min} of Corollary~\ref{corB} holds without its hypothesis on $\mu(f)$.}
\end{remark}

When the number of variables exceeds the characteristic, the conjecture follows from the estimates above:

\begin{proposition}\label{prop:conj_large_n}
If $n\geq p+1$, then Conjecture~\ref{conj:milnor} holds for every $f\in\mathfrak{m}^2$ with $\mu(f)<\infty$.
\end{proposition}

\begin{proof}
Let $\lambda$ be as in Conjecture~\ref{conj:milnor}, and put $B=\frac{(p+1)\lambda-(n-1)}{p}$. As $\lambda\leq1$, we have $B\leq\frac{p+2-n}{p}$, which is at most $0$ if $n\geq p+2$. So assume that $n=p+1$. Then $B\leq\frac1p$, and as $\fpt(f)\geq\frac{1}{\ord(f)}$ by Proposition 4.1 in~\cite{takagi2004f}, we may assume that $\ord(f)\geq p+1$. By Lemma~\ref{lem:milnor_min} with $w=(1,\dots,1)$, we have $\mu(f)\geq(\ord(f)-1)^n$, so $\lambda\leq\frac{p+1}{\ord(f)}$ and $B\leq\frac1p(\frac{(p+1)^2}{\ord(f)}-p)$, which is at most $\frac{1}{\ord(f)}$ if $\ord(f)\geq p+2$. Hence we may assume that $\ord(f)=p+1$. If the lowest-degree form of $f$ is not a power of a linear form, Proposition~\ref{prop:tangent_cone} gives $\fpt(f)\geq\frac1p\geq B$.

Otherwise, after a linear change of coordinates, the lowest-degree form of $f$ is $\kappa x_1^{p+1}$ with $\kappa\in k^\times$. Then $\partial_1(f)$ has order $p$, as its degree-$p$ part is $\kappa x_1^p$, while $\partial_j(f)$ has order at least $p+1$ for $j\geq2$. As in the proof of Lemma~\ref{lem:milnor_min}, the local intersection-multiplicity inequality gives $\mu(f)\geq p(p+1)^p$, so $\mu(f)^{1/n}\geq(p+1)r$ with $r=(\frac{p}{p+1})^{1/(p+1)}$. As $\fpt(f)\geq\frac1{p+1}$, it suffices to show that $B\leq\frac1{p+1}$, that is, $\lambda\leq\frac{p(p+2)}{(p+1)^2}$. Since $\lambda\leq\frac{p+1}{1+(p+1)r}$, this follows from
\[
    (p+1)^3\leq p(p+2)\bigl(1+(p+1)r\bigr).
\]
Using $\log(1-x)\geq-\frac{x}{1-x}$ with $x=\frac1{p+1}$, we get $\log r\geq-\frac{1}{p(p+1)}$, hence $r\geq1-\frac{1}{p(p+1)}$ and $1+(p+1)r\geq p+2-\frac1p$. Therefore
\[
    p(p+2)\bigl(1+(p+1)r\bigr)-(p+1)^3\geq(p+2)(p^2+2p-1)-(p+1)^3=p^2-3>0.\qedhere
\]
\end{proof}

\section{$F$-Pure Threshold and Semigroup of Values}

In this section, we prove Theorem~\ref{thm:semigr_fpt}, we give an elementary proof of the second case of Theorem~\ref{thm:conj_cases} for most irreducible plane curves, and we prove a lower bound for plane curves in terms of the Milnor and Tjurina numbers. Throughout this section, we assume that $n=2$ and denote $k[[x_1,x_2]]$ by $k[[x,y]]$.

\begin{definition}\label{def:semigroup}
Given an irreducible $f\in k[[x,y]]$, the \textbf{semigroup of values} of $f$, denoted $\Gamma(f)$, is the set of all values
\[
i(f,g)=\dim_k\left(\frac{k[[x,y]]}{\langle f,g\rangle}\right),
\]
where $g\in k[[x,y]]\setminus\langle f\rangle$.
The \emph{minimal generating sequence} $\beta_0,\beta_1,\ldots$ is defined
recursively by $\beta_0=\min(\Gamma(f)\setminus\{0\})$ and
$\beta_{j+1}=\min(\Gamma(f)\setminus\langle\beta_0,\ldots,\beta_j\rangle)$
for as long as this complement is nonempty. Its finiteness follows from
the numerical-semigroup theorem~\cite{hefez2003irreducible,barroso2012approach}.
\end{definition}

We summarize the main results about $\Gamma(f)$ that we use throughout this section in the following lemma:

\begin{lemma}\label{lem:semi_prope}
Let $f\in k[[x,y]]$ be an irreducible power series. Then:
\begin{enumerate}
    \item If $g,h\notin\langle f\rangle$, then $i(f,gh)=i(f,g)+i(f,h)$.
    \item If $g\notin\langle f\rangle$, then $i(f,g)\geq\ord(f)\ord(g)$.
    \item $\Gamma(f)$ is a numerical semigroup, i.e.\ $\gcd(\Gamma(f))=1$.
    \item $\min(\Gamma(f)\setminus\{0\})=\ord(f)$.
    \item There exists a finite sequence of integers $\beta_0,\dots,\beta_K$ that generates $\Gamma(f)$ as a semigroup and is minimal under inclusion.
\end{enumerate}
\end{lemma}

\begin{proof}
For item 1, see, for example, item 6 of Lemma 3.3 in~\cite{svoray2026Milnor}; for item 2, see Theorem 3 of Chapter 3.3 in~\cite{fulton2008algebraic}; for item 3, see Lemma 3.1 in~\cite{barroso2012approach}; and for item 4, see item 5 of Lemma 3.3 in~\cite{svoray2026Milnor}. For item 5, the generating sequence of Definition~\ref{def:semigroup} terminates by Proposition 6.1 in Chapter 6 of~\cite{hefez2003irreducible}, and its minimality and generation follow from the recursive construction.
\end{proof}

\begin{remark}\label{rem:2dim_sing}
\begin{enumerate}
    \item \textup{If $f\in k[[x,y]]$ is irreducible, then $\tau(f)<\infty$, since $f$ defines a reduced plane curve whose singular locus is supported at the closed point. This need not be true for the Milnor number. For example, if $\operatorname{char}(k)=p$, then $\mu(x^p+y^{p+1})=\infty$ but $\tau(x^p+y^{p+1})=p^2$.}
    \item \textup{After a linear change of coordinates, we may assume that $f_{\ord(f)}(0,1)\neq0$, so that $i(f,x)=\ord(f(0,y))=\ord(f)$; for irreducible $f$, this equals $\beta_0$ by item 4 of Lemma~\ref{lem:semi_prope}. This gives an alternative proof of $\fpt(f)\geq \frac{1}{\ord(f)}$ (see Proposition 4.1 in~\cite{takagi2004f}). Indeed, from $f^{\nu_e(f)+1}\in\langle x,y\rangle^{[p^e]}$ we may write $f^{\nu_e(f)+1}=gx^{p^e}+hy^{p^e}$. Reducing modulo $x$ gives $(\nu_e(f)+1)\ord(f)\geq p^e$, and hence $\frac{1}{\ord(f)}\leq\frac{\nu_e(f)+1}{p^e}\to\fpt(f)$.}
    \item \textup{The semigroup of values is an important tool in the study of curve singularities and of one-dimensional local rings more generally. For more information, see, for example,~\cite{assi2020numerical,barroso2012approach,zariski2006moduli}.}
    \item \textup{In item 5 of Lemma~\ref{lem:semi_prope}, $K>0$ if and only if $f\in\mathfrak{m}^2$, since a numerical semigroup has a unique generator if and only if it is $\mathbb{N}$, which is equivalent to $\ord(f)=1$.}
\end{enumerate}
\end{remark}

We first treat binomials, using the notation of Definition~\ref{def:basep}.

\begin{lemma}\label{lem:binomial}
Let $a,c\geq2$ be integers and $\lambda\in k^\times$. Set $s=\frac1a+\frac1c$ and $L=L_p(\frac1a,\frac1c)$, and, if $L<\infty$, set $\theta_a=\theta_L(\frac1a)$ and $\theta_c=\theta_L(\frac1c)$. Then:
\begin{enumerate}
    \item\label{binom:L} $L\geq1$. If $L<\infty$ and $\gcd(a,c)=1$, then $\theta_a+\theta_c>0$.
    \item\label{binom:exact} If $L=\infty$, then $\fpt(y^a-\lambda x^c)=s$.
    \item\label{binom:lower} If $L<\infty$, then $\fpt(y^a-\lambda x^c)\geq s-\frac{1-\theta_a-\theta_c}{p^L}$. If moreover $p\nmid ac$, then $\fpt(y^a-\lambda x^c)\geq s-\frac{1-s}{p}$.
    \item\label{binom:old} If $p\nmid a$, then $\fpt(y^a-\lambda x^c)\geq\frac1a+\frac{1}{a(c-1)}$ if $p\nmid c$, and $\fpt(y^a-\lambda x^c)\geq\frac1a+\frac{1}{ac}$ if $p\mid c$.
\end{enumerate}
\end{lemma}

\begin{proof}
After rescaling $x$, we may assume that $\lambda=-1$. Clearly $\fpt(y^a)=\frac1a$ and $\fpt(x^c)=\frac1c$. The first digit of $\frac1b$ is $\lceil\frac pb\rceil-1\leq\frac{p-1}{b}$, so the first digits of $\frac1a$ and $\frac1c$ have sum at most $(p-1)s\leq p-1$, and $L\geq1$. If $\gcd(a,c)=1$, then $a$ and $c$ cannot both divide $p^L$, so $\theta_a+\theta_c>0$ by Definition~\ref{def:basep}. This proves the first item.

By Lemma~\ref{lem:thom_sebastiani}, $\fpt(y^a+x^c)\geq s$ if $L=\infty$, and $\fpt(y^a+x^c)\geq s-\frac{1-\theta_a-\theta_c}{p^L}$ if $L<\infty$. Giving $x$ weight $a$ and $y$ weight $c$, Proposition~\ref{prop:weighted_fpt} gives $\fpt(y^a+x^c)\leq\frac{a+c}{ac}=s$, which proves the second item. If $p\nmid ac$, then $\theta_a\geq\frac1a$ and $\theta_c\geq\frac1c$ by Definition~\ref{def:basep}, and as $L\geq1$, the third item follows.

For the last item, note first that $s$ exceeds both bounds, since $(a-1)(c-1)\geq1$, so we may assume that $L<\infty$. Let $B=s-\frac{1-\theta_a-\theta_c}{p^L}$ and $P=p^L$, so that $PB=\lceil\frac Pa\rceil+\lceil\frac Pc\rceil-1$. As $a\nmid P$, $\lceil\frac Pa\rceil\geq\frac{P+1}{a}$, and as $\lceil\frac Pc\rceil\geq1$, $B\geq\frac1a+\frac1{aP}$. Also $\lceil\frac Pc\rceil\geq\frac{P+1}{c}$ if $p\nmid c$ and $\lceil\frac Pc\rceil\geq\frac Pc$ in general. Hence $B\geq\frac1a+\frac1c-\frac1P(1-\frac1a-\frac1c)$ if $p\nmid c$, and $B\geq\frac1a+\frac1c-\frac1P(1-\frac1a)$ if $p\mid c$. If $p\nmid c$, the first estimate gives the claim when $P\leq c-1$. When $P\geq c-1$, the second does, as $1-\frac1a-\frac1c\geq0$ and $\frac1a+\frac1c-\frac{1}{c-1}(1-\frac1a-\frac1c)=\frac1a+\frac{1}{a(c-1)}$. If $p\mid c$, the same comparison at $P=c$ gives $B\geq\frac1a+\frac{1}{ac}$.
\end{proof}

\noindent Lemma~\ref{lem:binomial} assumes neither $\gcd(a,c)=1$ nor $p\nmid ac$. When $p\nmid ac$, the bound of item~\ref{binom:lower} is the value given by the formula of Theorem 3.1 in~\cite{hernandez2015F}. Next, we need a weighted version of Hensel's lemma.

\begin{lemma}\label{lem:hensel}
Let $w=(w_1,w_2)\in\mathbb{N}_{>0}^2$ and $0\neq f\in k[[x,y]]$. Suppose that $\operatorname{in}_w(f)=GH$, where $G,H\in k[x,y]$ are coprime $w$-homogeneous polynomials of positive degrees. Then $f=\tilde G\tilde H$ with $\operatorname{in}_w(\tilde G)=G$ and $\operatorname{in}_w(\tilde H)=H$. In particular, $f$ is reducible.
\end{lemma}

\begin{proof}
We grade the ring $S=k[x,y]$ by setting $\deg(x)=w_1$ and $\deg(y)=w_2$. Let $\alpha=\deg(G)$ and $\beta=\deg(H)$. The common zeros of $G$ and $H$ in $k^2$ form a finite set, as $G$ and $H$ are coprime. This set is stable under $t\cdot(x,y)=(t^{w_1}x,t^{w_2}y)$, so it is $\{0\}$. Hence $G,H$ is a regular sequence, and the Hilbert series of $S/\langle G,H\rangle$ is $\frac{(1-t^\alpha)(1-t^\beta)}{(1-t^{w_1})(1-t^{w_2})}$, a polynomial of degree $\alpha+\beta-w_1-w_2$. So every $w$-homogeneous $F$ of degree $\delta\geq\alpha+\beta$ is $F=GB+HA$, with $A,B$ $w$-homogeneous of degrees $\delta-\beta,\delta-\alpha$.

Write $f=\sum_{\delta}f_\delta$ in $w$-homogeneous parts, so $f_{\alpha+\beta}=GH$. Set $G_0=G$, $H_0=H$, and for $s\geq1$ choose inductively $G_s,H_s$, of degrees $\alpha+s,\beta+s$, with $GH_s+HG_s=f_{\alpha+\beta+s}-\sum G_{s'}H_{s''}$, the sum over $s'+s''=s$ with $s',s''<s$. Since the weights are positive, $\tilde G=\sum_sG_s$ and $\tilde H=\sum_sH_s$ are power series with $f=\tilde G\tilde H$, and neither is a unit, as $\alpha,\beta>0$.
\end{proof}

The following lemma describes the first Newton edge of an irreducible plane curve in arbitrary characteristic.

\begin{lemma}\label{lem:branch_face}
Let $f\in\mathfrak{m}^2\subset k[[x,y]]$ be irreducible, and set $m=\ord(f)$. Then, after a change of coordinates, there are coprime integers $2\leq a<c$, an integer $d\geq1$ with $m=ad$, and $u,\lambda\in k^\times$ such that $\operatorname{in}_w(f)=u(y^a-\lambda x^c)^d$ for $w=(a,c)$.
\end{lemma}

\begin{proof}
The tangent cone $f_m$ is a product of linear forms, and by Lemma~\ref{lem:hensel} with $w=(1,1)$ it is not a product of two coprime forms of positive degree. So $f_m=u\ell^m$ with $\ell$ linear and $u\in k^\times$, and after a linear change of coordinates $\ell=y$. As $f$ is irreducible and $f\in\mathfrak{m}^2$, $y\nmid f$, so $f(x,0)\neq0$.

Every monomial $x^iy^j\neq y^m$ of $f$ has $i+j>m$, so $\frac{i}{m-j}>1$ if $j<m$. Let $\frac ca$, with $\gcd(a,c)=1$, be the minimum of $\frac{i}{m-j}$ over the monomials of $f$ with $j<m$. Then $c>a$, and for $w=(a,c)$ every monomial of $f$ has $w$-weight at least $cm$, with equality exactly for the monomials $x^{ct}y^{m-at}$, $t\geq0$, of $f$. Thus $\operatorname{in}_w(f)=\sum_{t=0}^T\kappa_tx^{ct}y^{m-at}$, with $\kappa_0=u$, $\kappa_T\neq0$ and $T\geq1$.

If $m>aT$, then $\operatorname{in}_w(f)=y^{m-aT}H$ with $H(x,0)=\kappa_Tx^{cT}\neq0$, contradicting Lemma~\ref{lem:hensel}. So $m=aT$. Set $d=T$. Since $\kappa_0\kappa_d\neq0$, the form $\sum_t\kappa_tX^tY^{d-t}$ factors as $u\prod_{s=1}^d(Y-\lambda_sX)$ with $\lambda_s\in k^\times$, so $\operatorname{in}_w(f)=u\prod_s(y^a-\lambda_sx^c)$. If $\lambda_s\neq\lambda_{s'}$, the factors $y^a-\lambda_sx^c$ and $y^a-\lambda_{s'}x^c$ are coprime, since their difference is a nonzero multiple of $x^c$ and $x\nmid y^a-\lambda_sx^c$. By Lemma~\ref{lem:hensel}, all $\lambda_s$ are equal.

If $a=1$, then $c\geq2$. The automorphism $y\mapsto y+\lambda x^c$ gives $f'(x,y)=f(x,y+\lambda x^c)$, irreducible of order $m$ with tangent cone $uy^m$. Every monomial of $f-\operatorname{in}_w(f)$ has $w$-weight greater than $cm$, so $\ord(f'(x,0))=\ord(f(x,\lambda x^c))>cm$, and we repeat.

If we always obtained $a=1$, we would get $c_1<c_2<\cdots$ and $\lambda_s\in k^\times$ such that $\varphi_s=\sum_{s'\leq s}\lambda_{s'}x^{c_{s'}}$ satisfies $\ord(f(x,\varphi_s))>c_sm$. Then $\varphi=\sum_s\lambda_sx^{c_s}$ satisfies $f(x,\varphi)\equiv f(x,\varphi_s)\pmod{x^{c_{s+1}}}$ for all $s$, so $f(x,\varphi)=0$. Applying the Weierstrass division theorem with respect to $y-\varphi$ shows that $y-\varphi$ divides $f$, which contradicts irreducibility, as $f\in\mathfrak{m}^2$. Hence $a\geq2$ after finitely many steps.
\end{proof}

The next lemma identifies the data of Lemma~\ref{lem:branch_face} with the first two generators of the semigroup of values.

\begin{lemma}\label{lem:face_semigroup}
Let $f\in\mathfrak{m}^2\subset k[[x,y]]$ be irreducible, and suppose that $\operatorname{in}_w(f)=u(y^a-\lambda x^c)^d$ for $w=(a,c)$, where $2\leq a<c$ are coprime, $d\geq1$ and $u,\lambda\in k^\times$. Then $\beta_0=ad$ and $\beta_1=cd$. In particular, $d=\gcd(\beta_0,\beta_1)$, $a=\frac{\beta_0}{d}$ and $c=\frac{\beta_1}{d}$.
\end{lemma}

\begin{proof}
Let $\mathcal{O}=\frac{k[[x,y]]}{\langle f\rangle}$, and let $\overline{\mathcal{O}}\cong k[[t]]$ be its normalization, which is a finite $\mathcal{O}$-module. For $h\in k[[x,y]]$, let $v(h)$ be the $t$-adic valuation of the image of $h$ in $\overline{\mathcal{O}}$, with $v(h)=\infty$ if $h\in\langle f\rangle$. For $h\notin\langle f\rangle$, multiplication by $h$ is injective on $\mathcal{O}$ and on $\overline{\mathcal{O}}$, and $\frac{\overline{\mathcal{O}}}{\mathcal{O}}$ has finite length, so the snake lemma gives
\[
    i(f,h)=\dim_k\frac{\mathcal{O}}{h\mathcal{O}}=\dim_k\frac{\overline{\mathcal{O}}}{h\overline{\mathcal{O}}}=v(h).
\]
Every monomial $x^iy^j$ of $f$ satisfies $ai+cj\geq acd$, and $y^{ad}$ and $x^{cd}$ occur in $f$ with coefficients $u$ and $u(-\lambda)^d$. Hence $v(x)=\ord(f(0,y))=ad$ and $v(y)=\ord(f(x,0))=cd$, so $v(h)\geq d\ord_w(h)$ for every $h\in k[[x,y]]$.

Since $a<c$, every $h\in\mathfrak{m}$ has $\ord_w(h)\geq a$, so $\beta_0=\min(\Gamma(f)\setminus\{0\})=v(x)=ad$. Now let $\gamma=v(h)\in\Gamma(f)$ with $\gamma<cd$, and set $W=\ord_w(h)$. Then $dW\leq\gamma<cd$, so $W<c$, and the only monomials of $w$-weight $W$ are powers of $x$. Thus $\operatorname{in}_w(h)=\kappa x^i$ with $ai=W$ and $\kappa\in k^\times$. Since $v(h-\kappa x^i)\geq d(W+1)>dW=v(\kappa x^i)$, we get $\gamma=dW=\beta_0i$. Hence every element of $\Gamma(f)$ smaller than $cd$ is a multiple of $\beta_0$, while $cd=v(y)\in\Gamma(f)$ is not, as $a\nmid c$. Therefore $\beta_1=cd$.
\end{proof}

For a first face of the shape found in Lemma~\ref{lem:branch_face}, the comparison of Proposition~\ref{prop:weighted_fpt} gives the following estimate, which does not require irreducibility.

\begin{proposition}\label{prop:face_sandwich}
Let $0\neq f\in\mathfrak{m}\subset k[[x,y]]$ with $\operatorname{in}_w(f)=u(y^a-\lambda x^c)^d$ for $w=(a,c)$, where $a,c\geq2$, $d\geq1$ and $u,\lambda\in k^\times$, and set $s=\frac1a+\frac1c$ and $L=L_p(\frac1a,\frac1c)$. Then $\frac1d\fpt(y^a+x^c)\leq\fpt(f)\leq\frac sd$. In particular, $\fpt(f)=\frac sd$ if $L=\infty$, which holds whenever $p\equiv1\pmod{\operatorname{lcm}(a,c)}$, and if $L<\infty$, then
\[
    0\leq\frac sd-\fpt(f)\leq\frac{1-\theta_L(\frac1a)-\theta_L(\frac1c)}{dp^L},
\]
where the right side is less than $\frac{1}{dp^L}\leq\frac1{dp}$ if $\gcd(a,c)=1$, and at most $\frac{1-s}{dp}$ if $p\nmid ac$.
\end{proposition}

\begin{proof}
Since $\ord_w(f)=acd$ and $\frac{a+c}{acd}=\frac sd\leq1$, the upper bound is Proposition~\ref{prop:weighted_fpt}, which also gives $\fpt(f)\geq\fpt(u(y^a-\lambda x^c)^d)$. Multiplication by the unit $u$ does not change $\fpt$, and for every $H\in\mathfrak{m}$ we have $\nu_e(H^d)=\lfloor\frac{\nu_e(H)}{d}\rfloor$, since membership of powers of $H$ in $\mathfrak{m}^{[q]}$ is monotone; hence $\fpt(H^d)=\frac{\fpt(H)}{d}$. Finally, the automorphism $x\mapsto\zeta x$ with $\zeta^c=-\lambda^{-1}$ takes $y^a-\lambda x^c$ to $y^a+x^c$, and the lower bound follows. The remaining assertions follow from Lemma~\ref{lem:binomial}. For the congruence, if $p\equiv1\pmod{\operatorname{lcm}(a,c)}$, then all base-$p$ digits of $\frac1a$ and $\frac1c$ equal $\frac{p-1}{a}$ and $\frac{p-1}{c}$, respectively, and their sum $(p-1)s$ is at most $p-1$, so $L=\infty$.
\end{proof}

\begin{proof}[Proof of Theorem~\ref{thm:semigr_fpt}]
By Lemma~\ref{lem:branch_face}, after a change of coordinates, which changes neither $\fpt(f)$ nor $\Gamma(f)$, we have $\operatorname{in}_w(f)=u(y^{a'}-\lambda x^{c'})^{d}$ for $w=(a',c')$, where $2\leq a'<c'$ are coprime, $d\geq1$ and $u,\lambda\in k^\times$. By Lemma~\ref{lem:face_semigroup}, $d=g$, $a'=a$ and $c'=c$. In particular, $2\leq a<c$, $\gcd(a,c)=1$, $\frac{1}{g}(\frac1a+\frac1c)=\frac1{\beta_0}+\frac1{\beta_1}$ and $ac=\frac{\beta_0\beta_1}{g^2}$, and $L\geq1$ by Lemma~\ref{lem:binomial}. Parts~\ref{cor:beta_bound}, \ref{cor:beta_gap}, \ref{cor:beta_equality} and~\ref{cor:beta_tame} now follow from Proposition~\ref{prop:face_sandwich}.
\end{proof}

\begin{remark}\label{rem:igusa_lct}
\textup{Igusa in Theorem 1 of~\cite{igusa1977first} proved that if $f\in\mathfrak m^2\subset\mathbb{C}[[x,y]]$ is irreducible, then $\lct(f)=\frac{1}{\beta_0}+\frac{1}{\beta_1}$. An alternative proof was provided by Kuwata in Theorem 1.1 of~\cite{kuwata1999log}. This result was later generalized by Lee in~\cite{lee2026log} to arbitrary algebraically closed fields. Thus Theorem~\ref{thm:semigr_fpt} can be viewed as an $F$-pure threshold analogue of Igusa's result, and by reduction modulo $p$ it gives another proof of it (see Corollary~\ref{cor:MTW}).}
\end{remark}

\begin{example}\label{ex:higher_terms_change_fpt}
\textup{The lower bound in Part~\ref{cor:beta_bound} of Theorem~\ref{thm:semigr_fpt} need not be attained. Let $k=\overline{\mathbb{F}}_2$ and $f=x^5+y^7+x^3y^3$. Give $x,y$ weights $7,5$, respectively. Then $\operatorname{in}_w(f)=x^5+y^7$, which is irreducible since $\gcd(5,7)=1$. Any nontrivial factorization of $f$ would induce a nontrivial factorization of its weighted initial form, so $f$ is irreducible. Exchanging $x$ and $y$, Lemma~\ref{lem:face_semigroup} with $(a,c,d)=(5,7,1)$ gives $\beta_0=5$ and $\beta_1=7$. By Theorem 3.1 of~\cite{hernandez2015F}, we have $\fpt(x^5+y^7)=\frac{1}{4}$, so Theorem~\ref{thm:semigr_fpt} gives $\frac14\leq\fpt(f)\leq\frac{12}{35}$. Put $I_q=\langle x^q,y^q\rangle$, where $q=2^e$. If $q=4^m$ with $m\geq0$, induction, starting with $q=1$, gives}
\[
f^{\frac{q-1}{3}}\equiv x^{q-1}y^{q-1}\pmod{I_q}.
\]
\textup{Indeed, if the congruence holds for $q$, then, in characteristic $2$,}
\[
f^{\frac{4q-1}{3}}
=f\big(f^{\frac{q-1}{3}}\big)^4
\equiv(x^5+y^7+x^3y^3)x^{4q-4}y^{4q-4}
\equiv x^{4q-1}y^{4q-1}\pmod{I_{4q}}.
\]
\textup{If instead $q=2\cdot4^m$, set $r=\frac{q}{2}=4^m$. Squaring the congruence for $r$ gives}
\[
f^{\frac{q-2}{3}}\equiv x^{q-2}y^{q-2}\pmod{I_q}.
\]
\textup{Thus, in either case, $f^{\lfloor\frac{q-1}{3}\rfloor}\notin I_q$, while multiplying the displayed surviving monomial by each term of $f$ shows that $f^{\lfloor\frac{q-1}{3}\rfloor+1}\in I_q$. Hence $\nu_e(f)=\lfloor\frac{q-1}{3}\rfloor$ for every $e$, and therefore $\fpt(f)=\frac13>\frac14=\fpt(x^5+y^7)$.}
\end{example}

\begin{remark}\label{rem:hernandez_arrigoni}
\begin{enumerate}
    \item \textup{If $p\nmid ac$, Theorem 3.1 of~\cite{hernandez2015F} computes $\fpt(y^a+x^c)$, and hence the left end of the interval in Part~\ref{cor:beta_bound} of Theorem~\ref{thm:semigr_fpt}, from the base-$p$ expansions of $\frac1a$ and $\frac1c$.}
    \item \textup{Arrigoni in~\cite{arrigoni2024f} studied the $F$-pure threshold of $f\in k[[x,y]]$, and more generally $c(f,\langle x^a,y^b\rangle)$, when $\Gamma(f)$ is generated by two elements, calling these ``simple algebroid plane branches.'' This class corresponds to $\gcd(\beta_0,\beta_1)=1$ in Theorem~\ref{thm:semigr_fpt}.}
\end{enumerate}
\end{remark}

We now give an elementary proof of Conjecture~\ref{conj:milnor} for most irreducible plane curves, with a sharper bound; the general case of plane curves is proved in Section~\ref{sec:plane}. We use the following elementary inequality.

\begin{lemma}\label{lem:convexity}
For real numbers $r_1,r_2\geq1$, we have $\frac{1}{1+r_1}+\frac{1}{1+r_2}\geq\frac{2}{1+\sqrt{r_1r_2}}$.
\end{lemma}

\begin{proof}
The function $\varphi(t)=\frac{1}{1+e^t}$ satisfies $\frac{d^2\varphi}{dt^2}=\frac{e^t(e^t-1)}{(1+e^t)^3}\geq0$ for $t\geq0$, so it is convex on $[0,\infty)$. Jensen's inequality at $t_1=\log r_1$ and $t_2=\log r_2$ gives the claim.
\end{proof}

\begin{proposition}\label{prop:conj_branches}
Let $f\in\mathfrak{m}^2\subset k[[x,y]]$ be irreducible with $\mu(f)<\infty$, let $\beta_0,\beta_1,g,a,c,L$ be as in Theorem~\ref{thm:semigr_fpt}, set $s=\frac1a+\frac1c$, and set $\lambda=\min\{1,\frac{2}{1+\sqrt{\mu(f)}}\}$. Assume that $p\mid a$, or that
\begin{equation}\label{eq:carry_hyp}
    L=\infty,\qquad\text{or}\qquad\frac{1-\theta_L(\frac1a)-\theta_L(\frac1c)}{p^L}\leq\frac{1-s}{p}.
\end{equation}
Condition~\eqref{eq:carry_hyp} holds, in particular, if $p\nmid ac$. Then
\[
    \fpt(f)\geq\lambda-\frac{1-\lambda}{p},
\]
so Conjecture~\ref{conj:milnor} holds for $f$. Moreover, if~\eqref{eq:carry_hyp} holds, then $\fpt(f)\geq\frac1{\beta_0}+\frac1{\beta_1}-\frac{1-s}{gp}$.
\end{proposition}

\begin{proof}
As in the proof of Theorem~\ref{thm:semigr_fpt}, after a change of coordinates, which changes neither $\fpt(f)$ nor $\mu(f)$, we have $\operatorname{in}_w(f)=u(y^a-\lambda'x^c)^g$ for $w=(a,c)$ and some $u,\lambda'\in k^\times$. Set $S=\frac1{\beta_0}+\frac1{\beta_1}=\frac sg$. Since $\ord_w(f)=acg>c$, Lemma~\ref{lem:milnor_min} gives $\mu(f)\geq\frac{(acg-a)(acg-c)}{ac}=(\beta_0-1)(\beta_1-1)\geq2$, so $\lambda=\frac{2}{1+\sqrt{\mu(f)}}<1$; note that $\lambda-\frac{1-\lambda}{p}$ is increasing in $\lambda$.

Suppose first that~\eqref{eq:carry_hyp} holds. If $p\nmid ac$, this is the case by Definition~\ref{def:basep}, since $\theta_L(\frac1a)\geq\frac1a$, $\theta_L(\frac1c)\geq\frac1c$ and $L\geq1$. By Proposition~\ref{prop:face_sandwich}, Lemma~\ref{lem:binomial} and~\eqref{eq:carry_hyp}, $\fpt(f)\geq\frac1g\fpt(y^a+x^c)\geq\frac1g(s-\frac{1-s}{p})=S-\frac{1-s}{gp}$.
Lemma~\ref{lem:convexity} with $r_1=\beta_0-1$ and $r_2=\beta_1-1$ gives $S\geq\frac{2}{1+\sqrt{(\beta_0-1)(\beta_1-1)}}\geq\lambda$. Therefore
\[
    \Big(S-\frac{1-s}{gp}\Big)-\Big(\lambda-\frac{1-\lambda}{p}\Big)
    =(S-\lambda)+\frac1p\Big((1-\lambda)-\frac{1-s}{g}\Big)\geq0,
\]
since $1-\lambda\geq1-S=1-\frac sg\geq\frac{1-s}{g}$.

Suppose next that $p\mid a$, and set $m=\ord(f)=\beta_0=ag$. Every monomial $x^iy^j$ of $f$ with $i+j=m$ satisfies $ai+c(m-i)\geq acg=cm$, so $i=0$, as $c>a$. Thus $f_m=uy^m$, and $\partial_y(f_m)=0$ because $p\mid a\mid m$. Hence both partial derivatives of $f$ lie in $\mathfrak{m}^m$, and Theorem 3 of Chapter 3.3 in~\cite{fulton2008algebraic} gives $\mu(f)=i(\partial_x(f),\partial_y(f))\geq m^2$. So $\lambda\leq\frac{2}{1+m}$, and since $m\geq a\geq p$,
\[
    \lambda-\frac{1-\lambda}{p}\leq\frac{2(p+1)}{p(1+m)}-\frac1p=\frac{2p+1-m}{p(1+m)}\leq\frac1m,
\]
where the last inequality is equivalent to $(m-1)(m-p)\geq0$. As $\fpt(f)\geq\frac1m$ by Proposition 4.1 in~\cite{takagi2004f}, the result follows.
\end{proof}

\begin{remark}\label{rem:conj_branches_open}
\textup{Proposition~\ref{prop:conj_branches} does not cover the case where $p\mid c$ and~\eqref{eq:carry_hyp} fails, which is treated, together with all other plane curves, in Section~\ref{sec:plane}. In this case $\partial_x(\operatorname{in}_w(f))=0$, so Lemma~\ref{lem:milnor_min} gives $\mu(f)>(\beta_0-1)(\beta_1-1)$, and this can be exactly what is needed: if $p=5$, $g=1$, $a=3$ and $c=5$, then~\eqref{eq:carry_hyp} fails and Lemma~\ref{lem:binomial} gives $\fpt(f)\geq\frac25$, while $\mu(f)\geq9$ gives $\lambda-\frac{1-\lambda}{p}\leq\frac25$.}
\end{remark}

Combining Example~\ref{ex:ADE} with Proposition~\ref{prop:weighted_fpt}, we obtain an upper bound in terms of the Tjurina number:

\begin{proposition}\label{prop:tau_upper}
Assume $p>5$ and let $f\in\mathfrak{m}^2\subset k[[x,y]]$ with $\tau(f)<\infty$. Then $\fpt(f)\leq\frac12+\frac{1}{\tau(f)}$. Moreover, $\fpt(f)\leq\frac12+\frac{1}{\tau(f)+1}$, unless $f\equivk x^2+y^{\tau(f)}$ with $p\mid\tau(f)$.
\end{proposition}

\begin{proof}
If $f$ is not of ADE type, then $\fpt(f)\leq\frac12$ by Lemma~\ref{lem:ADE_reduction} and Example~\ref{ex:ADE}. As $\fpt$ and $\tau$ are invariant under contact equivalence, we may assume that $f$ is one of the normal forms $x^2+y^m$ ($m\geq2$), $x^2y+y^m$ ($m\geq3$), $x^3+y^4$, $x^3+xy^3$, $x^3+y^5$. These are quasi-homogeneous for the weights $(m,2)$, $(m-1,2)$, $(4,3)$, $(3,2)$, $(5,3)$. So Proposition~\ref{prop:weighted_fpt} gives $\fpt(f)\leq\frac12+\frac1m$, $\frac12+\frac1{2m}$, $\frac7{12}$, $\frac59$, $\frac8{15}$, respectively. On the other hand, $\Jac(f)+\langle f\rangle$ equals $\langle x,my^{m-1},y^m\rangle$ for $x^2+y^m$, of colength $m-1$ if $p\nmid m$ and $m$ if $p\mid m$. For $x^2y+y^m$ it equals $\langle xy,x^2+my^{m-1},y^m\rangle$, since $y^m=(x^2y+y^m)-\frac x2\cdot2xy$, and it has colength $m+1$, with basis $1,x,y,\dots,y^{m-1}$. The last three normal forms have weighted degrees $12,9,15$, not divisible by $p$, so Euler's identity gives $f\in\Jac(f)$, and Lemma~\ref{lem:milnor_min} gives $\tau(f)=\mu(f)=6,7,8$. Comparing these values with the upper bounds proves both inequalities.
\end{proof}

\begin{remark}
\textup{The exception in Proposition~\ref{prop:tau_upper} is necessary. If $p>5$ and $f=x^2+y^{2p}$, then $\tau(f)=2p$. The base-$p$ digits of $\frac12$ and $\frac1{2p}$ add without carrying, so item~\ref{binom:exact} of Lemma~\ref{lem:binomial} gives $\fpt(f)=\frac12+\frac1{2p}>\frac12+\frac1{2p+1}$. Thus the first upper bound can be attained in the exceptional case.}
\end{remark}

\section{Critical Vectors of Binary Forms}\label{sec:critical}

In this section, we describe $F$-thresholds of products of linear forms in two variables with respect to ideals generated by two coprime forms. These results are due to Hern\'andez and Teixeira~\cite{hernandez2017f}, who phrase them in terms of syzygy gaps. We include self-contained proofs of the statements that we use, in a form that allows multiplicities and arbitrary characteristic.

Throughout this section, $R=k[X,Y]$ is graded by total degree, $U,V\in R$ are coprime homogeneous polynomials of positive degrees $A$ and $B$, $I=\langle U,V\rangle$, and $E=A+B$. Since $I$ is homogeneous and primary to $\langle X,Y\rangle$, membership of homogeneous polynomials in $I^{[q]}=\langle U^q,V^q\rangle$ may be tested in $R$ or in $k[[X,Y]]$, and we use the notation $\nu_e(G,I)$ and $c(G,I)$ of Definition~\ref{def:Fpt}. We fix
\[
    G=\prod_{i=1}^N\ell_i^{b_i},
\]
where the $\ell_i$ are pairwise nonproportional linear forms and $b_i\geq1$, and we set $D=\deg G=\sum_ib_i$, $S=\frac ED$ and $T=c(G,I)$.

\begin{definition}\label{def:critical}
Let $q=p^e$. A product $H=\prod_i\ell_i^{a_i}$ is \textbf{critical at level $q$} if $H\in I^{[q]}$ and $H/\ell_i\notin I^{[q]}$ whenever $a_i>0$. A vector $\mathbf{u}\in\mathbb{Z}[\frac1p]_{\geq0}^N$ is \textbf{critical} if $\prod_i\ell_i^{qu_i}$ is critical at level $q$ for some power $q$ of $p$ with $q\mathbf{u}\in\mathbb{Z}^N$. By Lemma~\ref{lem:relations}, this then holds for every such $q$.
\end{definition}

\begin{lemma}\label{lem:relations}
Let $q=p^e$.
\begin{enumerate}
    \item Every homogeneous polynomial of degree at least $Eq-1$ lies in $I^{[q]}$. In particular, a product that is critical at level $q$ has degree less than $Eq$.
    \item Every homogeneous $H\in I^{[q]}$ of degree $h<Eq$ has a unique representation $H=PU^q+QV^q$ with $P,Q$ homogeneous of degrees $h-Aq$ and $h-Bq$, where a polynomial of negative degree is zero.
    \item Let $H\in I^{[q]}$ be a product of the $\ell_i$ with $\deg H<Eq$, represented as in item 2. Then $H$ is critical at level $q$ if and only if $\gcd(P,Q)=1$.
    \item If $H$ is critical at level $q$, then $H^p$ is critical at level $pq$. Conversely, if $H_0^p$ is critical at level $pq$, then $H_0$ is critical at level $q$.
\end{enumerate}
\end{lemma}

\begin{proof}
The map $R_{Bq-1}\oplus R_{Aq-1}\to R_{Eq-1}$ given by $(P,Q)\mapsto PU^q+QV^q$ is injective: if $PU^q=-QV^q$, then $V^q\mid P$ as $\gcd(U,V)=1$, and $\deg P<Bq$ forces $P=Q=0$. Both sides have dimension $Eq$, so the map is surjective, and the first assertion of item 1 follows, as every monomial of degree at least $Eq-1$ is divisible by one of degree $Eq-1$. If $H$ is critical at level $q$, then $H\neq1$ since $1\notin I$, and if $\deg H\geq Eq$, then $H/\ell_i\in I^{[q]}$ for every $\ell_i\mid H$, which is impossible. Uniqueness in item 2 follows from the same divisibility argument, applied to $(P-P')U^q=(Q'-Q)V^q$ with $\deg(P-P')<Bq$.

For item 3, if $\gcd(P,Q)\neq1$, then some $\ell_i$ divides both $P$ and $Q$, since every common factor divides $H$. Thus $\ell_i\mid H$ and $H/\ell_i=\frac{P}{\ell_i}U^q+\frac{Q}{\ell_i}V^q\in I^{[q]}$. Conversely, if $\ell_i\mid H$ and $H/\ell_i=P'U^q+Q'V^q$, then $H=\ell_iP'U^q+\ell_iQ'V^q$, and uniqueness gives $\ell_i\mid\gcd(P,Q)$.

For item 4, if $H=PU^q+QV^q$ is critical at level $q$, then $H^p=P^pU^{pq}+Q^pV^{pq}$, where $\deg H^p<Epq$ and $\gcd(P^p,Q^p)=1$, so $H^p$ is critical at level $pq$ by item 3. Conversely, as $R$ is free over $k[X^p,Y^p]$, the Frobenius map is faithfully flat, so $H_0^p\in(I^{[q]})^{[p]}$ implies $H_0\in I^{[q]}$. If moreover $H_0/\ell_i\in I^{[q]}$ for some $\ell_i\mid H_0$, then $H_0^p/\ell_i=(H_0/\ell_i)^p\ell_i^{p-1}\in I^{[pq]}$, contradicting the criticality of $H_0^p$.
\end{proof}

\begin{lemma}\label{lem:comparison}
Let $H$ be critical at level $q$, and let $K\in I^{[q]}$ be a product of the $\ell_i$ such that $\deg\operatorname{lcm}(H,K)<Eq$. Then $H\mid K$.
\end{lemma}

\begin{proof}
Write $H=PU^q+QV^q$ and $K=CU^q+FV^q$ as in Lemma~\ref{lem:relations}, so that $\gcd(P,Q)=1$, and let $L=\operatorname{lcm}(H,K)$. Then
\[
    L=\frac LH\,PU^q+\frac LH\,QV^q=\frac LK\,CU^q+\frac LK\,FV^q,
\]
and uniqueness of the representation of $L$ gives $KP=HC$ and $KQ=HF$. Hence $H$ divides $\gcd(KP,KQ)=K$.
\end{proof}

\begin{theorem}\label{thm:critical}
We have $T\leq S$. If $T<S$, then there is a unique critical vector $\mathbf{u}$ with $u_i<Sb_i$ for every $i$, and
\[
    T=\max_i\frac{u_i}{b_i}.
\]
If moreover the least power $q$ of $p$ with $q\mathbf{u}\in\mathbb{Z}^N$ satisfies $q>1$, then the critical product $H=\prod_i\ell_i^{qu_i}$ is not a $p$-th power, $Eq-\deg H\leq N-2$, and
\[
    0<S-T\leq\frac{N-2}{qD}\leq\frac{N-2}{pD}.
\]
\end{theorem}

\begin{proof}
By Lemma~\ref{lem:bound_c_nu}, $\frac{\nu_e(G,I)+1}{p^e}$ is non-increasing with limit $T$, so $T=\inf\{\frac mq:q=p^e,\ G^m\in I^{[q]}\}$. Since $\deg G^m=mD$, Lemma~\ref{lem:relations} gives $G^m\in I^{[q]}$ whenever $mD\geq Eq-1$, and hence $T\leq S$.

Suppose that $T<S$, and choose $q$ and $m$ with $G^m\in I^{[q]}$ and $\frac mq<S$. Lowering the exponents of $G^m$ one at a time while the product remains in $I^{[q]}$, we reach a product $\prod_i\ell_i^{a_i}$ that is critical at level $q$, with $a_i\leq mb_i$. Thus $\mathbf{u}=\frac1q\mathbf{a}$ is critical, with $u_i\leq\frac mqb_i<Sb_i$. Set $t_0=\max_i\frac{u_i}{b_i}<S$. If $\frac{m'}{q'}>t_0$, then, at a common denominator $q''$, the critical product $\prod_i\ell_i^{q''u_i}$ divides $G^{m'q''/q'}$, so $G^{m'q''/q'}\in I^{[q'']}$ and $T\leq\frac{m'}{q'}$. Hence $T\leq t_0$. Conversely, suppose that $G^{m'}\in I^{[q']}$ with $\frac{m'}{q'}<t_0$, and pass to a common denominator $q''$, so that $G^{m'q''/q'}\in I^{[q'']}$. The exponent vectors $q''\mathbf{u}$ and $\frac{m'q''}{q'}\mathbf{b}$ are both coordinatewise strictly smaller than $q''S\mathbf{b}$, so the least common multiple of the two products has degree less than $q''SD=q''E$. Lemma~\ref{lem:comparison} gives $u_i\leq\frac{m'}{q'}b_i$ for every $i$, contradicting $\frac{m'}{q'}<t_0$. Hence $T=t_0$. If $\mathbf{v}$ is another critical vector with $v_i<Sb_i$ for every $i$, the same degree bound and Lemma~\ref{lem:comparison}, applied in both directions, show that the two critical products divide each other, so $\mathbf{u}=\mathbf{v}$.

For the gap estimate, let $q>1$ be least, write $H=\prod_i\ell_i^{a_i}=PU^q+QV^q$ with $\mathbf{a}=q\mathbf{u}$, set $h=\deg H<Eq$, let $n$ be the number of indices with $a_i>0$, and let $Z=H/\prod_{a_i>0}\ell_i$, so that $\deg Z=h-n$. By the minimality of $q$, some $a_i$ is prime to $p$, so $H$ is not a $p$-th power. As $q>1$, we have $H_X=P_XU^q+Q_XV^q$, and therefore
\[
    QH_X-Q_XH=\Omega_XU^q,\qquad PH_X-P_XH=-\Omega_XV^q,
    \qquad\text{where }\Omega_X=QP_X-PQ_X.
\]
Since $Z$ divides $H$ and $H_X$, it divides $\Omega_XU^q$ and $\Omega_XV^q$, and hence $\Omega_X$. Similarly, $Z$ divides $\Omega_Y=QP_Y-PQ_Y$. Neither $P$ nor $Q$ is zero, since otherwise $\gcd(P,Q)=1$ would make $H$ a nonzero multiple of $U^q$ or $V^q$, a $p$-th power. If $\Omega_X=\Omega_Y=0$, then $\gcd(P,Q)=1$ gives $P\mid P_X$, so $P_X=0$, and likewise $P_Y=Q_X=Q_Y=0$. As $k$ is perfect, $P$ and $Q$, and hence $H$, are then $p$-th powers, a contradiction. Thus $\Omega_X$ and $\Omega_Y$ do not both vanish. By Euler's identity, $X\Omega_X+Y\Omega_Y=(\deg P-\deg Q)PQ=q(B-A)PQ=0$, so $\Omega_X=YW$ and $\Omega_Y=-XW$ for a nonzero homogeneous $W$, and $Z\mid W$, as $Z$ divides $YW$ and $XW$. Comparing degrees,
\[
    h-n\leq\deg W=\deg P+\deg Q-2=2h-Eq-2,
\]
that is, $Eq-h\leq n-2\leq N-2$. Finally, $u_i\leq Tb_i$ gives $h=q\sum_iu_i\leq qTD$, so
\[
    S-T\leq\frac{Eq-h}{qD}\leq\frac{n-2}{qD}\leq\frac{N-2}{qD}.\qedhere
\]
\end{proof}

We apply Theorem~\ref{thm:critical} to weighted homogeneous polynomials, using the substitution of the proof of Lemma~\ref{lem:milnor_min}:

\begin{proposition}\label{prop:weighted_cover}
Let $g\in k[x,y]$ be $(a,c)$-homogeneous of degree $D$, where $a$ and $c$ are coprime, set $G=g(X^a,Y^c)$ and $I=\langle X^a,Y^c\rangle$, and write $c=c'\bar q$ with $p\nmid c'$ and $\bar q$ a power of $p$. Then $G$ is homogeneous of degree $D$, $\fpt(g)=c(G,I)$, and Theorem~\ref{thm:critical} applies with $E=a+c$ and $S=\frac{a+c}{D}$. If $x\nmid g$ and $c(G,I)<S$, then for every power $q$ of $p$ with $q\mathbf{u}$ integral, where $\mathbf{u}$ is the critical vector of Theorem~\ref{thm:critical}, the critical product at level $q$ equals $H_0(X^{c'},Y^{c'})$, up to a nonzero scalar, for some product $H_0\in k[Z,W]$ of linear forms. Assume moreover that $p\nmid a$, and write $g=x^\alpha y^\beta\prod_{i=1}^r(y^a-\eta_ix^c)^{d_i}$ with distinct $\eta_i\in k^\times$. Then, up to a nonzero scalar,
\[
    G=X^{a\alpha}Y^{c\beta}\prod_{i=1}^r\left(X^{ac'}-\xi_iY^{ac'}\right)^{\bar qd_i}
\]
with distinct $\xi_i\in k^\times$, each binomial has $ac'$ distinct linear factors, and $G$ has $N=ac'r+\mathbf{1}_{\alpha>0}+\mathbf{1}_{\beta>0}$ distinct linear factors. If $c(G,I)<S$, then the critical vector $\mathbf{u}$ of Theorem~\ref{thm:critical} is constant on the linear factors of each binomial, and its coordinates $u_X,u_Y$ at $X$ and $Y$ satisfy $\frac{u_X}{a},\frac{u_Y}{c'}\in\frac1q\mathbb{Z}$ for every power $q$ of $p$ with $q\mathbf{u}$ integral.
\end{proposition}

\begin{proof}
Every monomial of $g$ has $(a,c)$-weight $D$, so $G$ is homogeneous of degree $D$. The ring $k[X,Y]$ is free over $k[X^a,Y^c]$, with basis $\{X^iY^j:0\leq i<a,\ 0\leq j<c\}$, so for every $m$ and $q$, we have $g^m\in\langle x^q,y^q\rangle$ if and only if $G^m\in\langle X^{aq},Y^{cq}\rangle=I^{[q]}$. Hence $\fpt(g)=c(G,I)$, and $I$ is generated by coprime forms of degrees $a$ and $c$.

Suppose that $x\nmid g$ and $c(G,I)<S$, and let $H$ be the critical product at level $q$. As $g(0,y)\neq0$, we have $c\mid D$, and every monomial $x^iy^j$ of $g$ satisfies $ai=D-cj\equiv0\pmod c$, so $c\mid i$. Thus $G\in k[X^{c'},Y]$, and for a primitive $c'$-th root of unity $\zeta$, the automorphism $(X,Y)\mapsto(\zeta X,Y)$ preserves $G$ and $I$. It permutes the linear factors of $G$, preserving multiplicities, and maps critical products to critical products, so by the uniqueness in Theorem~\ref{thm:critical}, $\mathbf{u}$ is constant on its orbits. As $X\nmid G$, every linear factor of $G$ other than $Y$ is proportional to $Y-\theta X$ for some $\theta\neq0$, and its orbit has product $Y^{c'}-\theta^{c'}X^{c'}$. Hence $H=Y^{qu_Y}H_1(X^{c'},Y^{c'})$ for a product $H_1$ of linear forms, where $u_Y$ is the coordinate of $\mathbf{u}$ at $Y$, and $u_Y=0$ if $Y\nmid G$. If $u_Y>0$, then, as $H\in I^{[q]}$ and $H/Y\notin I^{[q]}$, and as $I^{[q]}=\langle X^{aq},Y^{cq}\rangle$ is a monomial ideal, some monomial of $H$ lies in $I^{[q]}$ while its quotient by $Y$ does not, so its $Y$-exponent is $cq$. As $c'$ divides $cq$, and all $Y$-exponents of $H$ are congruent to $qu_Y$ modulo $c'$, we get $c'\mid qu_Y$, and $H=H_0(X^{c'},Y^{c'})$ with $H_0=W^{qu_Y/c'}H_1$.

Now assume $p\nmid a$. Since $y^a-\eta x^c$ becomes $Y^{ac}-\eta X^{ac}=(Y^{ac'}-\eta^{1/\bar q}X^{ac'})^{\bar q}$, the displayed form of $G$ follows, with $\xi_i=\eta_i^{-1/\bar q}$ distinct. As $p\nmid ac'$, each binomial $X^{ac'}-\xi_iY^{ac'}$ has $ac'$ distinct linear factors, and different binomials have no common factor. If $\zeta$ is a primitive $ac'$-th root of unity, the automorphism $(X,Y)\mapsto(X,\zeta Y)$ preserves $I$, each binomial and the axes, and cyclically permutes the linear factors of each binomial, which have equal multiplicities. It maps critical products to critical products, so by the uniqueness in Theorem~\ref{thm:critical}, $\mathbf{u}$ is constant on these orbits. Thus the critical product at level $q$ is
\[
    H=X^{qu_X}Y^{qu_Y}\prod_i\left(X^{ac'}-\xi_iY^{ac'}\right)^{q\gamma_i}
\]
up to a nonzero scalar, and each of its monomials has the form $X^{qu_X+ac'i}Y^{qu_Y+ac'j}$. Suppose that $u_X>0$. As $H\in I^{[q]}$ and $H/X\notin I^{[q]}$, and as $I^{[q]}=\langle X^{aq},Y^{cq}\rangle$ is a monomial ideal, some monomial of $H$ lies in $I^{[q]}$ while its quotient by $X$ does not. Its $X$-exponent is therefore exactly $aq$, so $qu_X+ac'i=aq$ and $\frac{u_X}{a}=1-\frac{c'i}{q}$. Similarly, if $u_Y>0$, then $qu_Y+ac'j=cq$ for some $j$, so $\frac{u_Y}{c'}=\bar q-\frac{aj}{q}$.
\end{proof}

\section{The Milnor Number Bound for Plane Curves}\label{sec:plane}

In this section, we prove the second case of Theorem~\ref{thm:conj_cases}. Throughout, $n=2$, $R=k[[x,y]]$ and $\mathfrak m=\langle x,y\rangle$. For $M\geq1$, we set
\[
    \lambda(M)=\frac{2}{1+\sqrt M},\qquad
    B_p(M)=\frac{(p+1)\lambda(M)-1}{p}.
\]
As $\lambda(M)\leq1$, Conjecture~\ref{conj:milnor} for $n=2$ asserts that $\fpt(f)\geq B_p(\mu(f))$. Note that $B_p$ is decreasing and that $B_p(M)\leq\lambda(M)$.

\begin{theorem}\label{thm:plane}
Let $f\in\mathfrak{m}^2\subset k[[x,y]]$ with $\mu(f)<\infty$. Then $\fpt(f)\geq B_p(\mu(f))$.
\end{theorem}

We first collect the estimates that we use. Let $m=\ord(f)$, let $(a,c)$ be positive weights, let $g=\operatorname{in}_{(a,c)}(f)$ have weighted degree $D>\max\{a,c\}$, and set $S=\frac{a+c}{D}$. By Proposition 4.1 in~\cite{takagi2004f}, Proposition~\ref{prop:weighted_fpt} and Lemma~\ref{lem:milnor_min},
\begin{equation}\label{eq:pl_basic}
    \fpt(f)\geq\frac1m,\qquad
    \fpt(f)\geq\fpt(g),\qquad
    \mu(f)\geq\frac{(D-a)(D-c)}{ac},
\end{equation}
with equality in the last inequality if and only if $\mu(g)<\infty$. In particular, $\mu(f)\geq(m-1)^2$. If $S\leq1$, then
\begin{equation}\label{eq:pl_lambdaS}
    \lambda(\mu(f))\leq S.
\end{equation}
Indeed, setting $r_1=\frac aD$ and $r_2=\frac cD$, we have $\frac{(D-a)(D-c)}{ac}=1+\frac{1-S}{r_1r_2}\geq1+\frac{4(1-S)}{S^2}=(\frac2S-1)^2$, as $r_1r_2\leq\frac{S^2}{4}$. If moreover $\mu(g)<\infty$, then Proposition~\ref{prop:sqh_bound} and~\eqref{eq:pl_lambdaS} give
\begin{equation}\label{eq:pl_isolated}
    \fpt(f)\geq\min\left\{1,\frac{(p+1)S-1}{p}\right\}\geq B_p(\mu(f)),
\end{equation}
and this also holds if $S>1$, since then $\fpt(f)=1$. Finally, if the lowest-degree form $f_m$ of $f$ is not a power of a linear form, then Proposition~\ref{prop:tangent_cone} gives
\begin{equation}\label{eq:pl_nonpure}
    \fpt(f)\geq\frac{1}{m-1}.
\end{equation}

\begin{lemma}\label{lem:plane_reduction}
It suffices to prove Theorem~\ref{thm:plane} when $p\geq5$, $m=\ord(f)\leq p$ and $\mu(f)<p^2$.
\end{lemma}

\begin{proof}
As $\mu(f)\geq(m-1)^2$, we have $\lambda(\mu(f))\leq\frac2m$. If $m\geq p+2$, then $\fpt(f)\geq\frac1m\geq\frac1p(\frac{2(p+1)}{m}-1)\geq B_p(\mu(f))$. If $m=p+1$ and $f_m$ is not a power of a linear form, then~\eqref{eq:pl_nonpure} gives $\fpt(f)\geq\frac1p$, while $\mu(f)\geq p^2$ gives $B_p(\mu(f))\leq\frac1p$. Otherwise, after a linear change of coordinates, $f_m=\kappa y^{p+1}$ with $\kappa\in k^\times$. Then $\ord(\partial_y(f))=p$ and $\ord(\partial_x(f))\geq p+1$, so $\mu(f)\geq p(p+1)$ by Theorem 3 of Chapter 3.3 in~\cite{fulton2008algebraic}. The bound $\fpt(f)\geq\frac1{p+1}$ suffices if $\mu(f)\geq p(p+1)+1$, because $B_p(M)\leq\frac1{p+1}$ if and only if $\sqrt M\geq\frac{2p^2+2p+1}{2p+1}$, and
\[
    \left(\frac{2p^2+2p+1}{2p+1}\right)^2=p(p+1)+1-\frac{p^2+p}{(2p+1)^2}.
\]
If $\mu(f)=p(p+1)$, equality in this intersection-multiplicity bound forces the tangent cones of $\partial_x(f)$ and $\partial_y(f)$ to have no common factor, so the degree-$(p+1)$ part of $\partial_x(f)$ contains $x^{p+1}$. This part is the $x$-derivative of the degree-$(p+2)$ part of $f$, so this is impossible if $p=2$. For odd $p$, it forces a nonzero $x^{p+2}$ term in $f$, and the weights $(p+1,p+2)$ give the initial form $\kappa y^{p+1}+bx^{p+2}$ with $b\neq0$, which has an isolated singularity; we apply~\eqref{eq:pl_isolated}.

If $m\leq p$ and $\mu(f)\geq p^2$, then $\fpt(f)\geq\frac1m\geq\frac1p\geq B_p(\mu(f))$. It remains to consider $p\leq3$, $m\leq p$ and $\mu(f)<p^2$. If $p=2$, then $m=2$. If the quadratic part of $f$ is a square, then both partial derivatives of $f$ lie in $\mathfrak{m}^2$, so $\mu(f)\geq4$, which is excluded. Otherwise it is a product of two distinct linear forms, which has an isolated singularity, and~\eqref{eq:pl_isolated} with weights $(1,1)$ gives $\fpt(f)=1$.

If $p=3$ and $m=3$, a tangent cone that is a cube would force both partial derivatives into $\mathfrak{m}^3$, and hence $\mu(f)\geq9$, which is excluded. Otherwise~\eqref{eq:pl_nonpure} gives $\fpt(f)\geq\frac12$, which suffices as $B_3(5)<\frac12$ and $\mu(f)\geq5$. Indeed, $\mu(f)\geq4$, with equality only if $f_3$ has an isolated singularity by Lemma~\ref{lem:milnor_min}, and a cubic form in characteristic $3$ has none: by Euler's identity, $x\partial_x(f_3)=-y\partial_y(f_3)$, so its partial derivatives have a common linear factor or both vanish. If $p=3$ and $m=2$, then, by the splitting lemma (see Theorem 2.1 in~\cite{greuel2025splitting}), either the quadratic part of $f$ has rank two and $\fpt(f)=1$ by Remark~\ref{rem:fpt_contact}, or $f\equivr y^2+h(x)$, with $d=\ord(h)\geq3$ and $\mu(f)=\ord(h')$. If $3\nmid d$, the initial form $y^2+\kappa x^d$ has an isolated singularity, and~\eqref{eq:pl_isolated} applies. If $3\mid d$, then $\ord(h')\geq d$ and $\mu(f)<9$ give $d\in\{3,6\}$. If $d=6$, then $\mu(f)\geq6$ and $\fpt(f)\geq\frac12\geq B_3(\mu(f))$. If $d=3$, then $\mu(f)\geq3$, and Proposition~\ref{prop:weighted_fpt} and item~\ref{binom:old} of Lemma~\ref{lem:binomial} give $\fpt(f)\geq\fpt(y^2+\kappa x^3)\geq\frac23>B_3(3)$.
\end{proof}

The next lemma is a positive-characteristic version of Varchenko's adapted coordinates. For a power series $F$, let $\delta(F)$ denote the \textbf{diagonal distance} of its Newton polygon, that is, the number with $(\delta(F),\delta(F))$ on the boundary of the Newton polygon of $F$.

\begin{lemma}\label{lem:adapted}
Let $f\in\mathfrak{m}^2\subset k[[x,y]]$ with $\mu(f)<\infty$. After a change of coordinates, one of the following holds:
\begin{enumerate}
    \item\label{adapted:vertex} $\operatorname{in}_w(f)=\kappa(xy)^t$ for some positive weights $w$, some $t\geq1$ and some $\kappa\in k^\times$.
    \item\label{adapted:edge} There are coprime positive integers $a\leq c$ such that, up to a nonzero scalar,
    \begin{equation}\label{eq:pl_form}
        g=\operatorname{in}_{(a,c)}(f)=x^\alpha y^\beta\prod_{i=1}^r(y^a-\eta_ix^c)^{d_i}
    \end{equation}
    with distinct $\eta_i\in k^\times$ and $r\geq1$, and, setting $D=a\alpha+c\beta+ac\sum_id_i$ and $S=\frac{a+c}{D}$,
    \begin{equation}\label{eq:pl_balanced}
        S\alpha\leq1,\qquad S\beta\leq1,\qquad Sd_i\leq1\ \text{for every }i,\qquad S\leq1.
    \end{equation}
\end{enumerate}
\end{lemma}

\begin{proof}
Since $\mu(f)<\infty$, $f$ is reduced: if $h^2\mid f$ with $h\in\mathfrak m$, then $\Jac(f)\subset\langle h\rangle$, which is not $\mathfrak m$-primary. For every $F\in\mathfrak m^2$, all monomials of $F$ have degree at least $2$, so $\delta(F)\geq1$. If $(\delta(f),\delta(f))$ is a vertex of the Newton polygon of $f$, positive weights $w$ in the interior of its normal cone give the first case. Otherwise, $(\delta(f),\delta(f))$ lies in the interior of an edge with primitive normal $(a,c)$, and after interchanging $x$ and $y$ we may assume $a\leq c$. The initial form $g=\operatorname{in}_{(a,c)}(f)$ factors as in~\eqref{eq:pl_form}, and since the edge lies on the line $ai+cj=D$, we have $\delta(f)=\frac{D}{a+c}=\frac1S\geq1$. The monomials of $g$ have $x$-exponents at least $\alpha$ and $y$-exponents at least $\beta$, and the edge contains $(\delta(f),\delta(f))$, so $\alpha,\beta\leq\delta(f)$. If $a\geq2$, then $ac\geq a+c$, and $D\geq acd_i$ gives $d_i\leq\frac{D}{ac}\leq\delta(f)$. Thus~\eqref{eq:pl_balanced} can fail only if $a=1$ and some factor $y-\eta x^c$ has multiplicity $d>\delta(f)$; there is at most one such factor, as $D\geq c\sum_id_i$.

In that case, replace $y$ by $y+\eta x^c$, and let $f_1$ be the resulting series. This substitution respects the $(1,c)$-weight, so $\operatorname{in}_{(1,c)}(f_1)$ is obtained from $g$ by the same substitution, and its monomials have $y$-exponent at least $d$. Its edge of weight $D$ ends at $(D-cd,d)$, where $D-cd<(1+c)\delta(f)-c\delta(f)=\delta(f)<d$. Every other monomial of $f_1$ with $y$-exponent less than $d$ has weight greater than $D$. Thus the Newton polygon of $f_1$ crosses the diagonal at height $\delta(f_1)<d$, strictly beyond the line of weight $D$, so $\delta(f_1)>\delta(f)$. Moreover, by convexity, the edge of $f_1$ through $(\delta(f_1),\delta(f_1))$ is flatter than the edge of weight $D$: if it has primitive normal $(1,c_1)$, then $c_1>c$. We repeat the procedure with $f_1$, without interchanging variables.

Suppose that this does not terminate, so that we obtain exponents $c_1<c_2<\cdots$, coefficients $\eta_j\in k^\times$, multiplicities $d_j$ and diagonal distances $\delta_1<\delta_2<\cdots$, where $\delta_j$ and $d_j$ refer to the series before the $j$th substitution. By~\eqref{eq:pl_lambdaS}, $\delta_j\leq\frac{1+\sqrt{\mu(f)}}{2}$, so $\delta_j$ converges. Since $D_j=(1+c_j)\delta_j\geq c_jd_j$, we have $\delta_j<d_j\leq(1+\frac1{c_j})\delta_j$. Hence the integers $d_j$ converge, and $d_j=d$ for $j\gg0$, with $\delta_j<d$ and $d\geq2$. For $j\gg0$, we also have $c_j\geq d$, and then the $(1,c_j)$-initial form before the $j$th substitution is $\kappa_jx^{\gamma_j}(y-\eta_jx^{c_j})^d$ with $\gamma_j<d$: any further factor would give $D_j\geq\gamma_j+c_j(d+1)$, and $\gamma_j\geq d$ would give $D_j\geq d(1+c_j)$, and either would imply $\delta_j\geq d$. The substitutions converge to a formal coordinate change, and we let $F$ be $f$ in the limiting coordinates. The substitutions after the $j$th one have $(1,c_j)$-weight greater than $c_j$, so they do not change the $(1,c_j)$-initial form, and $\operatorname{in}_{(1,c_j)}(F)=\kappa_jx^{\gamma_j}y^d$ for $j\gg0$. If $F$ had a monomial $x^iy^l$ with $l<d$, then $i+c_jl\geq\gamma_j+c_jd$ would give $i\geq c_j$ for all $j\gg0$, which is impossible. Hence $y^d\mid F$, contradicting reducedness. Thus the procedure terminates, giving the second case, unless it reaches a vertex, giving the first.
\end{proof}

We now dispose of the first case of Lemma~\ref{lem:adapted}, and of the second case with $S=1$ or $\fpt(g)=S$. In the first case, $\fpt(f)\geq\fpt((xy)^t)=\frac1t$, and~\eqref{eq:pl_lambdaS} with $D=(w_1+w_2)t>\max_iw_i$ and $S=\frac1t$ gives $\frac1t\geq\lambda(\mu(f))\geq B_p(\mu(f))$. In the second case with $S=1$, all nonzero numbers among $\alpha$, $\beta$ and the $d_i$ are one, and $D=a+c\geq ac\sum_id_i$ forces $a=1$. Thus $g$ is $x(y-\eta x^c)$, $y(y-\eta x)$ or $(y-\eta_1x)(y-\eta_2x)$ up to a scalar, the quadratic part of $f$ is a product of two distinct linear forms, and~\eqref{eq:pl_isolated} with weights $(1,1)$ gives $\fpt(f)=1$. Finally, if $\fpt(g)=S$, then~\eqref{eq:pl_basic} and~\eqref{eq:pl_lambdaS} give $\fpt(f)\geq S\geq B_p(\mu(f))$.

For the remaining cases, we need the following extension of Proposition~\ref{prop:qh_bound} to some weighted homogeneous polynomials without an isolated singularity.

\begin{lemma}\label{lem:plane_monomial}
Let $a\leq c$ be coprime positive integers, let $r\geq1$ with $ar\geq2$, let $\eta_1,\dots,\eta_r\in k^\times$ be distinct, let $\alpha,\beta\geq0$, and set
\[
    g=x^\alpha y^\beta G_1,\qquad G_1=\prod_{i=1}^r(y^a-\eta_ix^c),\qquad D=a\alpha+c\beta+acr,\qquad S=\frac{a+c}{D}.
\]
Assume that $p\nmid acr$, $S\leq1$, $S\alpha\leq1$ and $S\beta\leq1$. Then $\fpt(g)\geq S-\frac{1-S}{p}$.
\end{lemma}

\begin{proof}
Grade $k[x,y]$ by $\deg(x)=a$ and $\deg(y)=c$, so that $G_1$ is homogeneous of degree $D_1=acr$. Let $\theta_1=\partial_x$ if $\alpha=0$ and $\theta_1=x\partial_x$ if $\alpha>0$, let $\theta_2=\partial_y$ if $\beta=0$ and $\theta_2=y\partial_y$ if $\beta>0$, and set $J=\langle\theta_1(G_1),\theta_2(G_1)\rangle$. These generators are homogeneous of degrees $d_1=D_1-a\mathbf{1}_{\alpha=0}$ and $d_2=D_1-c\mathbf{1}_{\beta=0}$, which are positive as $cr\geq ar\geq2$.

We claim that $\theta_1(G_1)$ and $\theta_2(G_1)$ have no common zero in $k^2\setminus\{0\}$. At a common zero $(x_0,y_0)$ with $x_0y_0\neq0$, both partial derivatives of $G_1$ vanish, so $G_1(x_0,y_0)=0$ by the weighted Euler identity $D_1G_1=ax\partial_x(G_1)+cy\partial_y(G_1)$, as $p\nmid D_1$. But the factors $y^a-\eta_ix^c$ are smooth away from the origin, as $p\nmid ac$, and any two of them meet only at the origin, so the curve $G_1=0$ has no singular point other than the origin. If $x_0=0\neq y_0$, then $G_1(0,y)=y^{ar}$, so $\theta_2(G_1)(0,y_0)$ is $ary_0^{ar-1}$ or $ary_0^{ar}$, which is nonzero as $p\nmid ar$. Similarly, if $y_0=0\neq x_0$, then $G_1(x,0)=\kappa x^{cr}$ with $\kappa\neq0$, and $\theta_1(G_1)(x_0,0)\neq0$, as $p\nmid cr$. This proves the claim. Hence $\theta_1(G_1),\theta_2(G_1)$ is a homogeneous regular sequence, and as in the proof of Proposition~\ref{prop:qh_bound}, every monomial of degree greater than $\sigma=d_1+d_2-a-c$ lies in $J$.

Fix $e\geq1$, set $q=p^e$, and let $N=\nu_e(g)+1$ be the least integer with $g^N\in\mathfrak{m}^{[q]}$. Put $L=d_1+d_2-D_1$. We show that
\begin{equation}\label{eq:pl_monomial}
    \frac Nq\geq S-\frac{L}{pD}.
\end{equation}
If $N\geq q$, $\alpha N\geq q$ or $\beta N\geq q$, then $\frac Nq\geq\min\{1,\frac1\alpha,\frac1\beta\}\geq S$, and~\eqref{eq:pl_monomial} holds, as $L\geq D_1-a-c\geq0$ since $ar\geq2$ and $a\leq c$. Otherwise, $k=q-\alpha N$ and $l=q-\beta N$ are positive, and as $\mathfrak{m}^{[q]}$ is a monomial ideal, $g^N\in\mathfrak{m}^{[q]}$ gives $G_1^N\in\mathfrak{a}=\langle x^k,y^l\rangle$. Let $N_1\leq N$ be the least integer with $G_1^{N_1}\in\mathfrak{a}$, and write $N_1=p^tN_0$ with $p\nmid N_0$. Then $p^t\leq N_1<q$, so $p^t\leq\frac qp$. Set $k'=\lceil\frac{k}{p^t}\rceil$, $l'=\lceil\frac{l}{p^t}\rceil$ and $\mathfrak{a}'=\langle x^{k'},y^{l'}\rangle$. The monomials of $(G_1^{N_0})^{p^t}$ are the $x^{p^ti}y^{p^tj}$ for the monomials $x^iy^j$ of $G_1^{N_0}$, so $G_1^{N_0}\in\mathfrak{a}'$. Moreover, $G_1^{N_0-1}\notin\mathfrak{a}'$, since otherwise $G_1^{N_1-p^t}\in\langle x^{p^tk'},y^{p^tl'}\rangle\subset\mathfrak{a}$.

The derivations $x\partial_x$ and $y\partial_y$ preserve every monomial ideal. If $\alpha=0$, then $k'=\frac{q}{p^t}$ is divisible by $p$, so $\partial_x$ preserves $\mathfrak{a}'$, and similarly $\partial_y$ preserves $\mathfrak{a}'$ if $\beta=0$. Hence $N_0G_1^{N_0-1}\theta_i(G_1)=\theta_i(G_1^{N_0})\in\mathfrak{a}'$ for $i=1,2$, and as $p\nmid N_0$, we get $G_1^{N_0-1}J\subset\mathfrak{a}'$. As $\mathfrak{a}'$ is a monomial ideal and $G_1^{N_0-1}\notin\mathfrak{a}'$, some monomial $x^iy^j$ of $G_1^{N_0-1}$ satisfies $i<k'$ and $j<l'$. Then $x^{k'-1-i}y^{l'-1-j}G_1^{N_0-1}$ contains the monomial $x^{k'-1}y^{l'-1}$, so it is not in $\mathfrak{a}'$, and therefore $x^{k'-1-i}y^{l'-1-j}\notin J$. Its degree is thus at most $\sigma$, and as $ai+cj=(N_0-1)D_1$, this gives $N_0D_1\geq ak'+cl'-L$. Multiplying by $p^t$ and using $p^tk'\geq k$, $p^tl'\geq l$ and $p^t\leq\frac qp$, we get
\[
    ND_1\geq N_1D_1\geq a(q-\alpha N)+c(q-\beta N)-\frac{qL}{p},
\]
that is, $ND\geq(a+c)q-\frac{qL}{p}$, which is~\eqref{eq:pl_monomial}.

Finally, $D-a-c-L=a(\alpha-\mathbf{1}_{\alpha>0})+c(\beta-\mathbf{1}_{\beta>0})\geq0$, so~\eqref{eq:pl_monomial} gives $\frac{\nu_e(g)+1}{p^e}\geq S-\frac{D-a-c}{pD}=S-\frac{1-S}{p}$ for every $e\geq1$. By Lemma~\ref{lem:bound_c_nu}, the left side converges to $\fpt(g)$.
\end{proof}

\begin{proposition}\label{prop:plane_tame}
In the second case of Lemma~\ref{lem:adapted}, assume that $S<1$, $\fpt(g)<S$, $p\geq5$, $m\leq p$, $\mu(f)<p^2$ and $p\nmid ac$. Then $\fpt(f)\geq B_p(\mu(f))$.
\end{proposition}

\begin{proof}
By Proposition~\ref{prop:weighted_cover}, $\fpt(g)=c(G,I)$ with $G=g(X^a,Y^c)$ and $I=\langle X^a,Y^c\rangle$, and $G$ has $N=acr+\mathbf{1}_{\alpha>0}+\mathbf{1}_{\beta>0}$ distinct linear factors. The critical vector of $G$ is not integral: its coordinates at the nonaxis factors are less than $Sd_i\leq1$, and its coordinates at $X$ and $Y$ are less than $Sa\alpha\leq a$ and $Sc\beta\leq c$, so an integral critical product would be a monomial $X^iY^j$ with $i<a$ and $j<c$, which is not in $I$. Hence Theorem~\ref{thm:critical} gives
\[
    \fpt(f)\geq\fpt(g)\geq S-\frac{N-2}{pD}.
\]
If $N-2\leq D-a-c$, this gives $\fpt(f)\geq S-\frac{1-S}{p}=\frac{(p+1)S-1}{p}\geq B_p(\mu(f))$, by~\eqref{eq:pl_lambdaS}. Put $d=\sum_id_i$. Then
\[
    N-2-(D-a-c)=\big(a(1-\alpha)+\mathbf{1}_{\alpha>0}-1\big)+\big(c(1-\beta)+\mathbf{1}_{\beta>0}-1\big)-ac(d-r).
\]
The first term is $a-1$ if $\alpha=0$ and at most $0$ otherwise, and similarly for the second. So this is negative if $d>r$, as $ac>a+c-2$, and it is at most $0$ if $a=c=1$, or if $a=r=1$ and $\beta>0$. If $a=r=1$ and $\beta=0$, then $S=\frac{1+c}{\alpha+c}$, and~\eqref{eq:pl_balanced} gives $\alpha\leq1$, so $S\geq1$, which is excluded.

It remains to treat the case where $d=r$, $a<c$ and $ar\geq2$. Then every $d_i$ equals $1$, and $g=x^\alpha y^\beta\prod_{i=1}^r(y^a-\eta_ix^c)$ up to a nonzero scalar. We show that $p\nmid r$. Every monomial $x^iy^j$ of $f$ satisfies $ai+cj\geq D$, so $i+j\geq\frac Dc\geq ar$, as $a<c$. If $p\mid r$, then $ar\geq p\geq m\geq\frac Dc\geq ar$, so $a=1$, $r=p$, $m=p$, and $\frac Dc=ar$ forces $\alpha=\beta=0$. Then the monomials of $f$ of degree $m$ have weight at least $D=cm$, so $f_m=\kappa y^p$ as $c>1$. Then both partial derivatives of $f$ lie in $\mathfrak{m}^p$, and $\mu(f)\geq p^2$, contrary to our assumption. Hence $p\nmid acr$, and by~\eqref{eq:pl_balanced} and $S<1$, Lemma~\ref{lem:plane_monomial} applies to $g$. With~\eqref{eq:pl_basic} and~\eqref{eq:pl_lambdaS}, it gives $\fpt(f)\geq\fpt(g)\geq\frac{(p+1)S-1}{p}\geq B_p(\mu(f))$.
\end{proof}

For the case $p\mid ac$, we need three elementary inequalities.

\begin{lemma}\label{lem:wild_numerics}
Let $p\geq5$ and $2\leq u\leq p$. Then:
\begin{enumerate}
    \item\label{wild:largeM} $B_p((u-1)(2p-1))\leq\frac1u$.
    \item\label{wild:largep} If $p\geq2u$, then $B_p(u(p-1)+1)\leq\frac1u$.
    \item\label{wild:round} If $u<p$ and $p=\kappa u+s$ with $1\leq s\leq u-1$, then $B_p((u-1)(p-1)+1)\leq\frac{\kappa+1}{p}$.
\end{enumerate}
\end{lemma}

\begin{proof}
If $y>0$ and $2p+1-py>0$, then $B_p(M)\leq y$ if and only if $M(1+py)^2\geq(2p+1-py)^2$. For item~\ref{wild:largeM} with $u\geq3$, writing $p=u+t$ with $t\geq0$, the difference of the two sides multiplied by $u^2$ is
\[
    4u^2(u^2-3u+1)+4ut(2u^2-4u+1)+u(6u-7)t^2+2(u-1)t^3\geq0.
\]
For $u=2$, the corresponding difference is $2(p^3-p^2-4p-4)>0$, as $p\geq5$. For item~\ref{wild:largep}, writing $p=2u+t$, the difference multiplied by $u^2$ is
\[
    u\left(2u^3-u^2+8u+5u^2t+6ut+4ut^2+t^3+3t^2+4t\right)>0.
\]
For item~\ref{wild:round}, the difference multiplied by $(\kappa+2)^2$ is a quadratic in $s$ with leading coefficient $-4$, so it suffices to check $s=1$ and $s=u-1$. The former gives $\kappa(u-1)(\kappa^2u+4u-8)\geq0$. At $s=u-1$, for $\kappa=1$ the result is $2(u-3)^2\geq0$. For $\kappa\geq2$, the result is a quadratic in $u$ with positive leading coefficient $\kappa^2(\kappa+1)$, with value $2\kappa^3>0$ at $u=2$ and derivative $3\kappa^3+\kappa^2-4\kappa-4>0$ there, so it is positive for $u\geq2$.
\end{proof}

\begin{proposition}\label{prop:plane_wild}
In the second case of Lemma~\ref{lem:adapted}, assume that $S<1$, $\fpt(g)<S$, $p\geq5$, $m\leq p$, $\mu(f)<p^2$ and $p\mid ac$. Then $\fpt(f)\geq B_p(\mu(f))$.
\end{proposition}

\begin{proof}
Every monomial $x^iy^j$ of $f$ satisfies $ai+cj\geq D$, so $i+j\geq\frac Dc$, as $a\leq c$. If $p\mid a$, then $p\geq m\geq\frac Dc\geq a\sum_id_i+\beta\geq a\geq p$, so $m=a=p$, $\alpha=\beta=0$ and $\sum_id_i=1$. Then the monomials of $f$ of degree $m$ have weight at least $D=cm$, so $f_m=\kappa y^p$, and $\mu(f)\geq p^2$, contrary to our assumption. Hence $p\nmid a$ and $p\mid c$, and we write $c=c'\bar q$ as in Proposition~\ref{prop:weighted_cover}. Set
\[
    d=\sum_id_i,\qquad u=\beta+ad\leq\frac Dc\leq m\leq p.
\]
We have $u\geq2$: if $u=1$, then $a=d=1$ and $\beta=0$, and as $g\in\mathfrak m^2$, also $\alpha\geq1$, so $S\geq1$.

By Proposition~\ref{prop:weighted_cover}, $N-2\leq ac'r\leq\frac{cu}{p}$, as $ar\leq ad\leq u$ and $\bar q\geq p$. Moreover, $D\geq cu$, $a\leq u$ and $c\geq p$ give $D-a-c\geq c(u-1)-a\geq\frac{cu}{p}$, since $c(u-1-\frac up)\geq p(u-1)-u\geq u$. Hence, if the critical vector of $G$ is not integral, Theorem~\ref{thm:critical} gives $\fpt(f)\geq S-\frac{N-2}{pD}\geq S-\frac{1-S}{p}\geq B_p(\mu(f))$, as in the proof of Proposition~\ref{prop:plane_tame}.

Suppose that the critical vector is integral. By Proposition~\ref{prop:weighted_cover}, its critical product has the form
\[
    H=X^iY^j\prod_\ell\left(X^{ac'}-\xi_\ell Y^{ac'}\right)^{e_\ell},
    \qquad i,j,e_\ell\in\mathbb{N}.
\]
Here $i<Sa\alpha\leq a$, and $\frac ia\in\mathbb{Z}$ by Proposition~\ref{prop:weighted_cover}, so $i=0$. The monomials of $H$ with positive $X$-exponent lie in $\langle X^a\rangle$, as $ac'\geq a$, and $H$ contains $Y^{j+ac'\sum_\ell e_\ell}$. Hence $H\in I$ is equivalent to
\begin{equation}\label{eq:pl_integer}
    j+ac'\sum_\ell e_\ell\geq c.
\end{equation}
By Theorem~\ref{thm:critical}, $\fpt(g)=t$, where
\[
    t=\max\left\{\frac{j}{c\beta},\ \frac{e_\ell}{\bar qd_\ell}:1\leq\ell\leq r\right\},
\]
and the first term is omitted if $\beta=0$. Since $j\leq tc\beta$ and $e_\ell\leq t\bar qd_\ell$,~\eqref{eq:pl_integer} gives $c\leq tc\beta+tacd=tcu$, so $t\geq\frac1u$. On the other hand, as $D\geq cu$, Lemma~\ref{lem:milnor_min} gives
\begin{equation}\label{eq:pl_wildmu}
    \mu(f)\geq(u-1)\left(\frac{cu}{a}-1\right).
\end{equation}
If $c\geq2p$ or $u\geq2a$, the right side is at least $(u-1)(2p-1)$, as $u\geq a$, and item~\ref{wild:largeM} of Lemma~\ref{lem:wild_numerics} gives $B_p(\mu(f))\leq\frac1u\leq t\leq\fpt(f)$.

Otherwise $c=p$ and $u<2a$. Thus $\bar q=p$, $c'=1$, $d=r=d_1=1$, and $u=a+\beta$ with $0\leq\beta<a$. In fact,
\begin{equation}\label{eq:pl_strictwild}
    \mu(f)>(u-1)\left(\frac{pu}{a}-1\right).
\end{equation}
Indeed, if $\alpha=0$, then $\partial_x(g)=0$, so $\mu(g)=\infty$ and the inequality in Lemma~\ref{lem:milnor_min} is strict, while if $\alpha>0$, then $D>cu$. If $u=p$, then $\frac Dc\leq m\leq p=u$ forces $D=cu$, so $\alpha=0$, and the monomials of $f$ of degree $m=u$ have weight at least $D=cm$, so $f_m=\kappa y^p$, which is again impossible. Hence $u<p$, and we write $p=\kappa u+s$ with $1\leq s\leq u-1$. Now~\eqref{eq:pl_integer} reads $j+ae\geq p$, and $pt=\max\{\frac j\beta,e\}$, with $\frac j\beta$ omitted if $\beta=0$. If $e\geq\kappa+1$, then $pt\geq\kappa+1$. If $e\leq\kappa$, then $j\geq p-a\kappa=\kappa\beta+s$, which is impossible if $\beta=0$, and gives $pt\geq\kappa+\frac s\beta$ otherwise. Thus
\[
    pt\geq\kappa+1\ \text{ if }\beta=0,\qquad
    pt\geq\kappa+\min\left\{1,\frac s\beta\right\}\ \text{ if }\beta>0.
\]
If $\beta=0$ or $\beta\leq s$, then $pt\geq\kappa+1$, and~\eqref{eq:pl_strictwild} gives $\mu(f)\geq(u-1)(p-1)+1$, so item~\ref{wild:round} of Lemma~\ref{lem:wild_numerics} gives $B_p(\mu(f))\leq\frac{\kappa+1}{p}\leq t$.

Suppose that $\beta>s$. Then $a=u-\beta\leq u-1$, so~\eqref{eq:pl_wildmu} gives $\mu(f)\geq u(p-1)+1$. If $\kappa\geq2$, then $p\geq2u$, and item~\ref{wild:largep} of Lemma~\ref{lem:wild_numerics} gives $B_p(\mu(f))\leq\frac1u\leq t$. If $\kappa=1$, put $\rho=1+\frac s\beta$, so that $pt\geq\rho$. As $\beta<\frac u2$, we have $\rho+1>2+\frac{2s}{u}=\frac{2p}{u}$, and hence
\[
    \frac{2(p+1)}{\rho+1}-1<\frac{u(p+1)}{p}-1<u,
\]
since $u<p$. As $\mu(f)\geq u(p-1)+1>u^2$, it follows that $\sqrt{\mu(f)}>\frac{2(p+1)}{\rho+1}-1$, that is, $B_p(\mu(f))<\frac\rho p\leq t$. In all cases, $\fpt(f)\geq\fpt(g)=t\geq B_p(\mu(f))$.
\end{proof}

\begin{proof}[Proof of Theorem~\ref{thm:plane}]
Both $\fpt(f)$ and $\mu(f)$ are invariant under coordinate changes. By Lemma~\ref{lem:plane_reduction}, we may assume that $p\geq5$, $m\leq p$ and $\mu(f)<p^2$. We apply Lemma~\ref{lem:adapted}. The first case, and the second case with $S=1$ or $\fpt(g)=S$, were treated after Lemma~\ref{lem:adapted}. In the remaining cases, Proposition~\ref{prop:plane_tame} applies if $p\nmid ac$, and Proposition~\ref{prop:plane_wild} applies if $p\mid ac$.
\end{proof}

\begin{proof}[Proof of Theorem~\ref{thm:conj_cases}]
Both $\fpt(f)$ and $\mu(f)$ are invariant under coordinate changes. In the first case, let $h=\operatorname{in}_w(f)$ and $D=\ord_w(f)$. As $f\in\mathfrak{m}^2$, we have $\mu(h)>0$, so by the first paragraph of the proof of Proposition~\ref{prop:qh_bound}, each $\partial_i(h)$ is nonzero of positive $w$-degree $D-w_i$. Thus $D>\max_iw_i$, Lemma~\ref{lem:milnor_min} gives $\mu(f)=\prod_i\frac{D-w_i}{w_i}$, and Proposition~\ref{prop:sqh_milnor} applies. The second case is Theorem~\ref{thm:plane}, and the third is Proposition~\ref{prop:conj_large_n}. If $p=2$, then either $n=2$ or $n\geq3=p+1$, which gives the last assertion.
\end{proof}

\section{Milnor and Tjurina Numbers of Plane Curves}\label{sec:MT}

In this section, we bound the $F$-pure threshold of a plane curve from below in terms of its order, Milnor number and Tjurina number. The main step, Proposition~\ref{prop:face_MT}, uses the critical vectors of Section~\ref{sec:critical}.

\begin{proposition}\label{prop:MT_curves}
Let $f\in\mathfrak{m}^2\subset k[[x,y]]$ with $\tau(f)<\infty$, and assume that $p\nmid\ord(f)$. Then
\[
    \fpt(f)\geq\frac{1}{\ord(f)}+\frac{1}{\mu(f)+\tau(f)},
\]
where we set $\frac{1}{\mu(f)+\tau(f)}=0$ if $\mu(f)=\infty$. Moreover, if the lowest-degree form of $f$ is not a power of a linear form, then $\fpt(f)\geq\frac{1}{\ord(f)-1}$, and this holds without assuming $p\nmid\ord(f)$.
\end{proposition}

For irreducible $f$, Theorem~\ref{thm:semigr_fpt} gives a sharper estimate. The next two examples show that the bound of Proposition~\ref{prop:MT_curves} is attained, and that the hypothesis $p\nmid\ord(f)$ cannot be dropped.

\begin{example}\label{ex:MT_sharp}
\textup{Let $p$ be odd, let $q=p^e$ with $e\geq1$, and let $f=x^2+y^{q+1}$. Then $\Jac(f)=\langle x,y^q\rangle$ and $f=x\cdot x+y\cdot y^q\in\Jac(f)$, so $\mu(f)=\tau(f)=q$. Every monomial of $f^N$ other than $x^{2N}$ is divisible by $y^{q+1}$, so $f^N\in\mathfrak{m}^{[q]}$ if and only if $2N\geq q$. Hence $\nu_e(f)=\frac{q-1}{2}$, and since $\frac{\nu_e(f)+1}{p^e}$ is non-increasing by Lemma~\ref{lem:bound_c_nu}, $\fpt(f)\leq\frac{\nu_e(f)+1}{q}=\frac12+\frac{1}{2q}=\frac{1}{\ord(f)}+\frac{1}{\mu(f)+\tau(f)}$. So equality holds in Proposition~\ref{prop:MT_curves}. Note that $f=(x+\iota y^{\frac{q+1}{2}})(x-\iota y^{\frac{q+1}{2}})$ with $\iota^2=-1$ is reducible. Equality also holds for $f=x^2+y^2$, for which $\fpt(f)=1$ and $\mu(f)=\tau(f)=1$.}
\end{example}

\begin{example}\label{ex:MT_p_divides}
\textup{Let $f=x^p+x^{p+1}+y^{p+1}$, so that $\ord(f)=p$. Then $\Jac(f)=\langle x^p,y^p\rangle$ and $f=(1+x)x^p+y\cdot y^p\in\Jac(f)$, so $\mu(f)=\tau(f)=p^2$. As $f\in\mathfrak{m}^{[p]}$, we have $\nu_1(f)=0$, and Lemma~\ref{lem:bound_c_nu} gives $\fpt(f)\leq\frac{\nu_1(f)+1}{p}=\frac1p$. Hence $\fpt(f)=\frac{1}{\ord(f)}<\frac{1}{\ord(f)}+\frac{1}{\mu(f)+\tau(f)}$.}
\end{example}

The key arithmetic input is the following rounding estimate.

\begin{lemma}\label{lem:rounding}
Let $s\geq2$, let $b_1,\dots,b_s\geq1$ and $A_1,\dots,A_s\geq0$ be integers, and set $\delta=\sum_jb_j$, $Q=\sum_jA_j$ and $t=\max_j\frac{A_j}{b_j}$. If $\gcd(Q,\delta)=1$, then
\[
    \delta t-Q\geq\frac{s-1}{\delta-1}.
\]
\end{lemma}

\begin{proof}
As $\delta\geq s\geq2$ and $\gcd(Q,\delta)=1$, we have $Q>0$, so $t>0$. Write $t=\frac uv$ in lowest terms, and set $r_j=ub_j-vA_j$. These are nonnegative integers with $\sum_jr_j=u\delta-vQ=v(\delta t-Q)$, and some $r_j$ vanishes, as $t$ is a maximum. If $r_j=0$, then $v\mid ub_j$, so $v\mid b_j$. Not every $r_j$ vanishes: otherwise $t\delta=Q$, so $v=\delta$ as $\gcd(Q,\delta)=1$, and then $\delta\mid b_j$ for every $j$, which is impossible as $b_j<\delta$. Let $z$ be the number of $j$ with $r_j=0$, so that $1\leq z\leq s-1$. Then $\delta\geq vz+(s-z)$ and $\sum_jr_j\geq s-z$, so
\[
    v(\delta t-Q)(\delta-1)\geq(s-z)(vz+s-z-1)\geq vz(s-z)\geq v(s-1).\qedhere
\]
\end{proof}

The next proposition treats a single Newton face whose initial form is not divisible by $x$.

\begin{proposition}\label{prop:face_MT}
Let $a<c$ be coprime positive integers, and let
\[
    g=y^\beta\prod_{i=1}^r(y^a-\eta_ix^c)^{d_i}\in k[x,y],
\]
where $\beta\geq0$, $r\geq1$, $d_i\geq1$, and the $\eta_i\in k^\times$ are distinct. Assume that $\ord(g)\geq3$, that $p\nmid\ord(g)$, and, if $a=1$, that the coefficient of $x^cy^{\ord(g)-1}$ in $g$ is zero. Then
\[
    \fpt(g)\geq\frac{c}{c\ord(g)-a}=\frac{1}{\ord(g)}+\frac{a}{\ord(g)(c\ord(g)-a)},
\]
unless $g=y^{\ord(g)}+\lambda x^\gamma$ for some $\lambda\in k^\times$ and some $\gamma$ divisible by $p$.
\end{proposition}

\begin{proof}
As $a<c$, the lowest-degree form of $g$ is $y^{\ord(g)}$, so $\ord(g)=\beta+a\sum_id_i$, and $g$ is $(a,c)$-homogeneous of degree $c\ord(g)$. Put $S=\frac{a+c}{c\ord(g)}$, and note that $\frac{c}{c\ord(g)-a}<S$, as $c\ord(g)>a+c$.

We first treat binomials and powers of binomials. If $g=y^{\ord(g)}+\lambda x^\gamma$ with $p\nmid\gamma$, then $\frac{\gamma}{\ord(g)}=\frac ca$, and item~\ref{binom:old} of Lemma~\ref{lem:binomial} gives $\fpt(g)\geq\frac{1}{\ord(g)}+\frac{1}{\ord(g)(\gamma-1)}=\frac{c}{c\ord(g)-a}$. Suppose that $g=(y^a-\eta_1x^c)^{d_1}$ with $d_1\geq2$. If $a=1$, the coefficient of $x^cy^{\ord(g)-1}$ is $-d_1\eta_1\neq0$, as $d_1=\ord(g)$ is prime to $p$. Hence $a\geq2$, and $p\nmid a$, as $a$ divides $\ord(g)$. By Proposition~\ref{prop:face_sandwich} and item~\ref{binom:old} of Lemma~\ref{lem:binomial},
\[
    \fpt(g)\geq\frac{1}{d_1}\Big(\frac1a+\frac{1}{ac}\Big)=\frac{1}{\ord(g)}+\frac{1}{c\ord(g)}\geq\frac{1}{\ord(g)}+\frac{1}{\ord(g)(cd_1-1)}=\frac{c}{c\ord(g)-a}.
\]
So we may assume that $g$ is neither a binomial nor a power of a binomial. Then $a<\ord(g)$, and $\beta>0$ or $r\geq2$.

Write $c=c'\bar q$ with $p\nmid c'$ and $\bar q$ a power of $p$, and set
\[
    F=W^\beta\prod_{i=1}^r\big(W^a-\eta_i^{1/\bar q}Z^a\big)^{d_i}\in k[Z,W].
\]
This is a form of degree $\ord(g)$ with $F(0,1)=1$, and $G=g(X^a,Y^c)=F(X^{c'},Y^{c'})^{\bar q}$, as $Y^{ac}-\eta_iX^{ac}=(Y^{ac'}-\eta_i^{1/\bar q}X^{ac'})^{\bar q}$. The coefficient of $ZW^{\ord(g)-1}$ in $F$ vanishes. Indeed, if $a\geq2$, all $Z$-exponents of $F$ are divisible by $a$, and if $a=1$, this coefficient is $-\sum_id_i\eta_i^{1/\bar q}=-(\sum_id_i\eta_i)^{1/\bar q}$, where $-\sum_id_i\eta_i$ is the coefficient of $x^cy^{\ord(g)-1}$ in $g$. Write $F=\prod_{j=1}^sL_j^{b_j}$ with pairwise nonproportional linear forms $L_j$ satisfying $L_j(0,W)=W$, so that $\sum_jb_j=\ord(g)$ and $s\geq2$. Each $L_j=W-\theta Z$ with $\theta\neq0$ contributes the $c'$ linear factors of $Y^{c'}-\theta X^{c'}$ to $G$, each of multiplicity $\bar qb_j$, and $L_j=W$ contributes $Y$, of multiplicity $cb_j$. Hence $G$ has $N=c'(s-\epsilon)+\epsilon\leq c'\ord(g)$ distinct linear factors, where $\epsilon=1$ if $\beta>0$ and $\epsilon=0$ otherwise.

By Proposition~\ref{prop:weighted_cover}, $T=\fpt(g)=c(G,I)$ with $I=\langle X^a,Y^c\rangle$, and Theorem~\ref{thm:critical} applies with $E=a+c$, $D=c\ord(g)$ and $S$ as above. Suppose, for a contradiction, that
\begin{equation}\label{eq:MT_assume}
    T<\frac{c}{c\ord(g)-a}.
\end{equation}
Then $T<S$. Let $\mathbf{u}$ be the critical vector of Theorem~\ref{thm:critical}, let $q$ be the least power of $p$ with $q\mathbf{u}$ integral, and let $H$ be the critical product at level $q$. As $x\nmid g$, Proposition~\ref{prop:weighted_cover} gives $H=H_0(X^{c'},Y^{c'})$ for a product $H_0$ of linear forms in $Z,W$, up to a nonzero scalar. As $H$ divides a power of $G$, and nonproportional linear forms $L$ give coprime polynomials $L(X^{c'},Y^{c'})$, we may take $H_0=\prod_jL_j^{A_j}$ with $A_j\in\mathbb{Z}_{\geq0}$. Set
\[
    Q=\bar qq,\qquad h=\deg H_0=\sum_jA_j,\qquad t=\max_j\frac{A_j}{b_j}.
\]
If $L_j\neq W$, the linear factors of $L_j(X^{c'},Y^{c'})$ have exponent $A_j$ in $H$ and multiplicity $\bar qb_j$ in $G$, and if $L_j=W$, then $Y$ has exponent $c'A_j$ in $H$ and multiplicity $c'\bar qb_j$ in $G$. Thus Theorem~\ref{thm:critical} gives $T=\frac tQ$.

The ideal $I^{[q]}=\langle X^{aq},Y^{cq}\rangle$ is monomial, and a monomial $Z^iW^j$ is mapped to $X^{c'i}Y^{c'j}\in I^{[q]}$ if and only if $i\geq\frac{aQ}{c}$ or $j\geq Q$. Hence
\begin{equation}\label{eq:MT_member}
    H_0\in\langle Z^\ell,W^Q\rangle,\qquad\text{where }\ell=\Big\lceil\frac{aQ}{c}\Big\rceil.
\end{equation}
As $H_0(0,W)=W^h$, this gives $h\geq Q$; write $h=Q+\rho$. Since $h\leq t\ord(g)=QT\ord(g)$, \eqref{eq:MT_assume} gives
\begin{equation}\label{eq:MT_rho_up}
    \rho<\frac{aQ}{c\ord(g)-a}.
\end{equation}
If $q>1$, Theorem~\ref{thm:critical} gives $(a+c)q-c'h\leq N-2$. Multiplying by $\bar q$ and using $\bar qN\leq c\ord(g)$, we get
\begin{equation}\label{eq:MT_gap}
    aQ\leq c\rho+\bar q(N-2)\leq c\big(\rho+\ord(g)\big)-2\bar q.
\end{equation}

\emph{Step 1: $\rho=0$.} If $q=1$, then $Q=\bar q\leq c$, and $a(\bar q+1)\leq ac+a<c(a+1)\leq c\ord(g)$, so $\rho=0$ by~\eqref{eq:MT_rho_up}. Let $q>1$. Then~\eqref{eq:MT_rho_up} and~\eqref{eq:MT_gap} give $\rho(c\ord(g)-a)<c(\rho+\ord(g))-2\bar q$, that is,
\[
    \rho\big(c(\ord(g)-1)-a\big)<c\ord(g)-2\bar q\leq2\big(c(\ord(g)-1)-a\big),
\]
where the last inequality amounts to $c(\ord(g)-2)\geq2a-2\bar q$, which holds as $c>a$ and $\ord(g)\geq\max\{3,a+1\}$. Hence $\rho\leq1$. Suppose that $\rho=1$. Then $q>1$, and~\eqref{eq:MT_rho_up} gives
\begin{equation}\label{eq:MT_rho_one}
    aQ>c\ord(g)-a.
\end{equation}
If $p\mid a$, then $\bar q=1$, and $\beta>0$ as $p\nmid\ord(g)$. Writing $a=p^\nu a_0$ with $p\nmid a_0$, each $W^a-\eta_iZ^a=(W^{a_0}-\eta_i^{1/p^\nu}Z^{a_0})^{p^\nu}$ has $a_0$ distinct linear factors, so $N=ca_0r+1$. As $\ord(g)\geq1+ar\geq2+a_0r$, \eqref{eq:MT_gap} gives $aQ\leq c(1+a_0r)-1\leq c(\ord(g)-1)-1<c\ord(g)-a$, contradicting~\eqref{eq:MT_rho_one}. If $p\nmid a$, then $N=c'ar+\epsilon$, and~\eqref{eq:MT_gap} gives $aQ\leq c(1+ar)-\bar q(2-\epsilon)$. With $d=\sum_id_i\geq r$, we have
\[
    c\ord(g)-a-c(1+ar)+\bar q(2-\epsilon)=c\big(\beta+a(d-r)-1\big)-a+\bar q(2-\epsilon).
\]
If $d>r$, this is at least $(c-1)a+\bar q$ when $\beta>0$, and at least $(c-1)(a-1)-1+2\bar q$ when $\beta=0$. If $d=r$ and $\beta\geq2$, it equals $c(\beta-1)-a+\bar q$. In all three cases it is positive, contradicting~\eqref{eq:MT_rho_one}. Hence $d_i=1$ for every $i$ and $\beta\leq1$, so $F$ is squarefree, $s=\ord(g)$, every $b_j$ equals $1$, and $t$ is an integer. Now $t\ord(g)\geq h=Q+1$, while~\eqref{eq:MT_assume} and~\eqref{eq:MT_gap} give
\[
    t\ord(g)=QT\ord(g)<Q+\frac{aQ}{c\ord(g)-a}<Q+2,
\]
as $aQ\leq c(1+\ord(g))-2\bar q<2(c\ord(g)-a)$. Thus $t\ord(g)=h$, and as $A_j\leq t$ for every $j$, we get $A_j=t$ for every $j$, that is, $H_0=F^t$. As $p\mid Q$ and $t\ord(g)=Q+1$, we have $p\nmid t$. By~\eqref{eq:MT_rho_one}, $\frac{aQ}{c}>\ord(g)-1$, so $\ell\geq\ord(g)$. Since $\deg H_0=Q+1$, \eqref{eq:MT_member} forces the coefficient of $Z^jW^{Q+1-j}$ in $H_0$ to vanish for $2\leq j<\ell$, and for $j=1$ this coefficient is $t$ times the coefficient of $ZW^{\ord(g)-1}$ in $F$, which is zero. Thus $F(Z,1)^t-1\in Z^\ell k[Z]$. As $p\nmid t$, the lowest nonconstant terms of $F(Z,1)$ and $F(Z,1)^t$ have the same degree, and $F(Z,1)\neq1$, as $F$ is squarefree of degree $\ord(g)\geq3$. As $F(Z,1)$ has degree at most $\ord(g)\leq\ell$, we get $F=W^{\ord(g)}+\lambda Z^{\ord(g)}$ for some $\lambda\in k^\times$. Then $G=Y^{c\ord(g)}+\lambda^{\bar q}X^{c\ord(g)}$, so $g$ is a binomial, a contradiction. Hence $\rho=0$.

\emph{Step 2: $aQ\leq c(s-1)$.} Suppose first that $q=1$, so $Q=\bar q$. If $\bar q=1$, the claim follows from $a<c$ and $s\geq2$. If $\bar q>1$, then $p\nmid a$, so $s=ar+\epsilon\geq a+1$, as $r\geq2$ or $\beta>0$, and $c(s-1)\geq ca\geq a\bar q$. Now let $q>1$. By Theorem~\ref{thm:critical}, $H$ is not a $p$-th power, and hence neither is $H_0$. As $\deg H_0=Q$, \eqref{eq:MT_member} gives $H_0=W^Q+Z^\ell K_1$ for a form $K_1$. By Euler's identity, $Z\partial_Z(H_0)+W\partial_W(H_0)=QH_0=0$, so $\partial_Z(H_0)=WK$ and $\partial_W(H_0)=-ZK$ for a form $K$ of degree $Q-2$, and $K\neq0$, since $k$ is perfect and $H_0$ is not a $p$-th power. As $\partial_W(H_0)=Z^\ell\partial_W(K_1)$, we get $Z^{\ell-1}\mid K$. Moreover, $H_0/\operatorname{rad}(H_0)$ divides $\partial_Z(H_0)$ and $\partial_W(H_0)$, hence $WK$ and $ZK$, and hence $K$. As $Z\nmid H_0$, it follows that $Z^{\ell-1}H_0/\operatorname{rad}(H_0)$ divides $K$. Comparing degrees, $\ell-1+Q-s_+\leq Q-2$, where $s_+\leq s$ is the number of $j$ with $A_j>0$. Hence $\frac{aQ}{c}\leq\ell\leq s-1$.

\emph{Step 3: conclusion.} By Step 1, $\sum_jA_j=Q$. As $Q$ is a power of $p$ and $p\nmid\ord(g)$, Lemma~\ref{lem:rounding} applies, and together with Step 2 it gives
\begin{align*}
    T-\frac{1}{\ord(g)}=\frac{t\ord(g)-Q}{\ord(g)Q}&\geq\frac{s-1}{\ord(g)(\ord(g)-1)Q}\\
    &\geq\frac{a}{c\ord(g)(\ord(g)-1)}>\frac{a}{\ord(g)(c\ord(g)-a)},
\end{align*}
where the last inequality holds as $a<c$. The right-hand side equals $\frac{c}{c\ord(g)-a}-\frac{1}{\ord(g)}$, which contradicts~\eqref{eq:MT_assume}.
\end{proof}

\begin{proof}[Proof of Proposition~\ref{prop:MT_curves}]
Recall that $\tau(f)\leq\mu(f)$. Also, $\fpt$, $\mu$, $\tau$ and $\ord$ do not change under automorphisms of $k[[x,y]]$, and $\fpt$ and $\tau$ do not change under contact equivalence (Proposition~\ref{prop:mt_basic} and Remark~\ref{rem:fpt_contact}).

Suppose first that the lowest-degree form of $f$ is not a power of a linear form. By Proposition~\ref{prop:tangent_cone}, $\fpt(f)\geq\frac{1}{\ord(f)-1}$; this part does not use the assumption $p\nmid\ord(f)$. Moreover, $\Jac(f)+\langle f\rangle\subset\mathfrak{m}^{\ord(f)-1}$, so $\tau(f)\geq\dim_kk[[x,y]]/\mathfrak{m}^{\ord(f)-1}=\frac{\ord(f)(\ord(f)-1)}{2}$, and $\mu(f)+\tau(f)\geq2\tau(f)\geq\ord(f)(\ord(f)-1)$. Hence
\[
    \frac{1}{\ord(f)}+\frac{1}{\mu(f)+\tau(f)}\leq\frac{1}{\ord(f)}+\frac{1}{\ord(f)(\ord(f)-1)}=\frac{1}{\ord(f)-1}\leq\fpt(f).
\]

Suppose next that $\ord(f)=2$ and that the quadratic part of $f$ is the square of a linear form, so that $p\neq2$. By the splitting lemma (see Theorem 2.1 in~\cite{greuel2025splitting}), $f\equivr y^2+v(x)x^N$ for some $N\geq3$ and a unit $v\in k[[x]]$, where $N<\infty$ since $\tau(f)<\infty$. Choose $s\in k[[x]]$ with $s^2=v^{-1}$, which exists since $p\neq2$. Then $v^{-1}(y^2+vx^N)=(sy)^2+x^N$, so $f\equivk y^2+x^N$. As $\Jac(y^2+x^N)+\langle y^2+x^N\rangle=\langle y,Nx^{N-1},x^N\rangle$, we get $\tau(f)=N-1$ if $p\nmid N$ and $\tau(f)=N$ if $p\mid N$. In both cases, item~\ref{binom:old} of Lemma~\ref{lem:binomial} with $a=2$ and $c=N$ gives $\fpt(f)=\fpt(y^2+x^N)\geq\frac12+\frac{1}{2\tau(f)}\geq\frac12+\frac{1}{\mu(f)+\tau(f)}$.

Finally, suppose that $\ord(f)\geq3$ and that the lowest-degree form of $f$ is a power of a linear form. After a linear change of coordinates, it is a nonzero multiple of $y^{\ord(f)}$. If $\mu(f)=\infty$, the claim is $\fpt(f)\geq\frac{1}{\ord(f)}$, which holds by Proposition 4.1 in~\cite{takagi2004f}. So assume that $\mu(f)<\infty$. By the Weierstrass preparation theorem, $f=vF_0$ with $v$ a unit and
\[
    F_0=y^{\ord(f)}+\sum_{j<\ord(f)}\varphi_j(x)y^j,\qquad\varphi_j\in\langle x\rangle.
\]
Since the lowest-degree form of $F_0$ is $y^{\ord(f)}$, we have $\ord(\varphi_{\ord(f)-1})\geq2$. As $p\nmid\ord(f)$, the coordinate change $y\mapsto y-\varphi_{\ord(f)-1}(x)/\ord(f)$ preserves the lowest-degree form and removes the term $\varphi_{\ord(f)-1}(x)y^{\ord(f)-1}$ of $F_0$. Applying it to $f$, we may assume that $\varphi_{\ord(f)-1}=0$. Since $\tau(f)<\infty$, $f$ is reduced, so $F_0\neq y^{\ord(f)}$. Every monomial $x^iy^j$ of $F_0$ with $j<\ord(f)$ satisfies $i>\ord(f)-j$. Let $\frac ca$, with $a,c$ coprime, be the minimum of $\frac{i}{\ord(f)-j}$ over these monomials, and set $w=(a,c)$. Then $c>a$ and $\ord_w(F_0)=c\ord(f)$, and $g=\operatorname{in}_w(F_0)$ is $(a,c)$-homogeneous, contains $y^{\ord(f)}$ and at least one other monomial, and is not divisible by $x$. Hence
\[
    g=y^\beta\prod_{i=1}^r(y^a-\eta_ix^c)^{d_i}
\]
with $r\geq1$, $d_i\geq1$ and distinct $\eta_i\in k^\times$. Moreover, $\ord(g)=\ord(f)$, and $g$ has no monomial $x^cy^{\ord(f)-1}$, as $\varphi_{\ord(f)-1}=0$. So Proposition~\ref{prop:face_MT} applies to $g$. As $\operatorname{in}_w(f)=v(0)g$, Remark~\ref{rem:fpt_contact} and Proposition~\ref{prop:weighted_fpt} give $\fpt(f)=\fpt(F_0)\geq\fpt(g)$.

Since $D=\ord_w(f)=c\ord(f)>c$, Lemma~\ref{lem:milnor_min} gives
\[
    \mu(f)\geq\frac{(c\ord(f)-a)(c\ord(f)-c)}{ac}=\frac{(\ord(f)-1)(c\ord(f)-a)}{a}.
\]
Moreover, $\partial_x(f)$, $\partial_y(f)$ and $f$ have $w$-order at least $c(\ord(f)-1)$, so $\Jac(f)+\langle f\rangle$ is contained in the ideal generated by the monomials of $w$-weight at least $c(\ord(f)-1)$. The monomials outside this ideal are the $x^iy^j$ with $ai+cj<c(\ord(f)-1)$. For $j=\ord(f)-1-r$ with $1\leq r\leq\ord(f)-1$, there are $\lceil\frac{cr}{a}\rceil$ of them, and there are none with $j\geq\ord(f)-1$. Hence
\begin{equation}\label{eq:MT_tau}
    \tau(f)\geq\sum_{r=1}^{\ord(f)-1}\Big\lceil\frac{cr}{a}\Big\rceil\geq\frac{c\ord(f)(\ord(f)-1)}{2a}.
\end{equation}
Combining the two estimates,
\[
    \mu(f)+\tau(f)-\frac{\ord(f)(c\ord(f)-a)}{a}\geq\frac{c\ord(f)(\ord(f)-3)+2a}{2a}>0.
\]
Therefore, if $\fpt(g)\geq\frac{c}{c\ord(f)-a}$, then
\[
    \fpt(f)\geq\frac{1}{\ord(f)}+\frac{a}{\ord(f)(c\ord(f)-a)}\geq\frac{1}{\ord(f)}+\frac{1}{\mu(f)+\tau(f)}.
\]

By Proposition~\ref{prop:face_MT}, it remains to consider the case $g=y^{\ord(f)}+\lambda x^\gamma$ with $\lambda\in k^\times$ and $p\mid\gamma$, where $\gamma=\frac{c\ord(f)}{a}>\ord(f)$. Item~\ref{binom:old} of Lemma~\ref{lem:binomial} gives $\fpt(f)\geq\fpt(g)\geq\frac{1}{\ord(f)}+\frac{1}{\ord(f)\gamma}$, so it suffices to show that $\mu(f)+\tau(f)\geq\ord(f)\gamma$. As $\partial_x(g)=\lambda\gamma x^{\gamma-1}=0$, we have $\mu(\operatorname{in}_w(f))=\infty$, so the inequality of Lemma~\ref{lem:milnor_min} is strict. Since $\frac{c\ord(f)-a}{a}=\gamma-1$, it gives $\mu(f)\geq(\ord(f)-1)(\gamma-1)+1$. Also, \eqref{eq:MT_tau} reads $\tau(f)\geq\sum_{r=1}^{\ord(f)-1}\lceil\frac{\gamma r}{\ord(f)}\rceil$. If $\ord(f)\geq4$, these give
\[
    \mu(f)+\tau(f)-\ord(f)\gamma\geq\frac{(\ord(f)-3)\gamma}{2}-\ord(f)+2>0,
\]
as $\gamma\geq\ord(f)+1$. Let $\ord(f)=3$. If $3\nmid\gamma$, then $\lceil\frac\gamma3\rceil+\lceil\frac{2\gamma}3\rceil=\gamma+1$, and $\mu(f)+\tau(f)\geq(2\gamma-1)+(\gamma+1)=3\gamma$. If $3\mid\gamma$, then $w=(1,\frac\gamma3)$. As $\partial_x(\operatorname{in}_w(f))=0$, $\partial_x(f)$ has $w$-order at least $\gamma$, while $\partial_y(f)$ has $w$-order at least $\frac{2\gamma}{3}$. So Lemma~\ref{lem:milnor_min} gives $\mu(f)\geq\gamma\cdot\frac{2\gamma/3}{\gamma/3}=2\gamma$, while $\tau(f)\geq\lceil\frac\gamma3\rceil+\lceil\frac{2\gamma}3\rceil=\gamma$. In all cases, $\mu(f)+\tau(f)\geq3\gamma=\ord(f)\gamma$.
\end{proof}

\begin{remark}\label{rem:MT_tschirnhaus}
\textup{The coordinate change removing $\varphi_{\ord(f)-1}$ in the proof of Proposition~\ref{prop:MT_curves} is needed. In characteristic $7$, let $f=z^3-x^6z$ with $z=y-x^2$, so that $\mu(f)=\tau(f)=16$. The $(1,2)$-initial form of $f$ is $(y-x^2)^3$, whose $F$-pure threshold is $\frac13<\frac13+\frac1{32}=\frac{1}{\ord(f)}+\frac{1}{\mu(f)+\tau(f)}$. After $y\mapsto y+x^2$, we get $f=y^3-x^6y$, and Proposition~\ref{prop:face_MT} with $(a,c)=(1,3)$ gives $\fpt(f)\geq\frac38$.}
\end{remark}

\section{Log Canonical Threshold and Reduction mod $p$}\label{sec:lct}

In this section we apply the previous results to the log canonical threshold: we prove Corollary~\ref{cor:MTW}, and we record further bounds in Corollary~\ref{corB}. For more information on log canonical thresholds, see Chapter 9.3B in~\cite{lazarsfeld2003positivity} or~\cite{kollar2008powers,blum2016log}. The connection between contact equivalence, the Milnor number of a complex power series, and the log canonical threshold has also been studied, for example, in~\cite{almiron2022limit,mustactua2026some}.

\begin{definition}\label{def:lct}
Given a nonzero power series $f\in\mathfrak{m}\subset\mathbb{C}[[\underline{x}]]$, the \textbf{log canonical threshold} of $f$, denoted $\lct(f)$, is
\[
    \lct(f)=\inf_E\frac{A(E)}{\ord_E(f)},
\]
where $E$ runs over divisorial valuations centered at the origin and $A(E)$ denotes the log discrepancy of $E$.
\end{definition}

\begin{remark}
\textup{If $f$ is analytically convergent in a neighborhood of the origin in $\mathbb{C}^n$, this definition of $\lct$ agrees with the analytic characterization
\[
\lct(f)=\sup\left\{\lambda\geq0:\ |f|^{-\lambda}\text{ is locally }L^2
\text{ near }0\in\mathbb{C}^n\right\}.
\]
Equivalently, $\int_U |f|^{-2\lambda}\,d\mu<\infty$ for some neighborhood $U$ of the origin.}
\end{remark}

\begin{definition}[Reduction modulo primes]\label{def:reductions}
For $f\in\mathbb Z[[\underline{x}]]$ and a prime $p$,
$f_p\in\overline{\mathbb F}_p[[\underline{x}]]$ is its coefficientwise
reduction. We put $\mu_p(f)=\mu(f_p)$ and $\tau_p(f)=\tau(f_p)$;
$\mu(f)$ and $\tau(f)$ without subscripts are computed over $\mathbb C$.
\end{definition}

We use the following standard comparison theorem for polynomial reductions
(see~\cite{hara2003generalization,schwede2024singularities,mustata2012impanga}):
\begin{proposition}\label{prop:lct_fpt}
Let $0\neq f\in\mathfrak m\subset\mathbb Z[\underline{x}]$. For every sufficiently large prime $p$, we have $\fpt(f_p)\leq\lct(f)$. Moreover, as $p$ ranges over the primes,
\[
    \lim_{p\to\infty}\fpt(f_p)=\lct(f).
\]
\end{proposition}

For the specialization of Milnor and Tjurina numbers, we use the semicontinuity results of Section 2.2 in~\cite{greuel2021semicontinuity}:

\begin{proposition}\label{prop:tau_p}
Let $f\in\mathbb{Z}[[\underline{x}]]$. Then:
\begin{enumerate}
    \item If $\mu(f)<\infty$, then $\mu(f)=\mu_p(f)$ for every $p\gg0$.
    \item If $\mu_p(f)<\infty$ for some $p$, then $\mu(f)\leq\mu_p(f)$ and $\mu_q(f)\leq\mu_p(f)$ for every $q\gg0$.
\end{enumerate}
\end{proposition}

\begin{remark}\label{rem:tau_p}
\textup{An analogous result to Proposition~\ref{prop:tau_p} holds for $\tau(f)$ and $\tau_p(f)$, with the proof appearing in the same reference.}
\end{remark}

\begin{lemma}\label{lem:uniform_truncation}
Let $f\in\mathfrak m^2\subset\mathbb Z[[\underline{x}]]$ with
$\mu(f)<\infty$. There is a polynomial truncation $F\in\mathbb Z[\underline{x}]$
of $f$ such that $f\equivr F$ over $\mathbb C$, and
$f_p\equivr F_p$ over $\overline{\mathbb F}_p$ for all sufficiently large $p$.
In particular, $\lct$, $\fpt$ (where defined), $\mu$, $\tau$, and the
Brian\c{c}on--Skoda exponent are unchanged by the indicated substitutions.
For each fixed positive weight $w$, $F$ may also be chosen so that
$\operatorname{in}_w(F)=\operatorname{in}_w(f)$.
\end{lemma}
\begin{proof}
Set $d=\ord(f)$ and $M=\mu(f)$. By Proposition~\ref{prop:tau_p},
$\mu(f_p)=M$ for all sufficiently large $p$, and $\ord(f_p)=d$ for
all sufficiently large $p$ since the degree-$d$ initial form has at least
one nonzero integer coefficient. By right determinacy
(Proposition~\ref{prop:mt_basic} and~\cite{boubakri2012invariants}),
a uniform bound is $N=2M-d+2$. Choose $F$ to contain all ordinary-degree
terms of $f$ through degree $N$ (and enough additional terms to retain the
specified weighted initial form, if applicable).
Then $f-F\in\mathfrak m^{N+1}$ over $\mathbb C$, and the same inclusion
holds after reduction modulo every $p$. The right-determinacy theorem
therefore gives the asserted equivalences. The invariance of the local
thresholds follows from their definitions and contact invariance,
while that of the Jacobian invariants and Brian\c{c}on--Skoda exponent
follows from the chain rule; for $\lct$ see~\cite{lazarsfeld2003positivity}.
\end{proof}

\noindent By Lemma~\ref{lem:uniform_truncation}, the remaining proofs may
be carried out for polynomial $f$ without changing the conclusions for
formal power series. We do this throughout the rest of the section.

We obtain the following consequences for the log canonical threshold:

\begin{corollary}\label{corB}
Let $f\in\mathfrak{m}^2\subset\mathbb{Z}[[\underline{x}]]$, viewed also as an element of $\mathbb{C}[[\underline{x}]]$, and assume that $\mu(f)<\infty$. Then:
\begin{enumerate}
    \item\label{cor:lct_milnor} $\lct(f)\geq\frac{1}{\tau(f)}-\frac{1}{\mu(f)}$.
    \item\label{cor:lct_weighted} If $w\in\mathbb{N}_{>0}^n$ and $\mu(\operatorname{in}_w(f))<\infty$, then $\lct(f)=\min\{1,\frac{w_1+\cdots+w_n}{\ord_w(f)}\}$.
    \item\label{cor:lct_milnor_min} If $n\geq2$ and $\mu(f)=\prod_{i=1}^n\frac{D-w_i}{w_i}$ for some $w\in\mathbb{N}_{>0}^n$ with $D=\ord_w(f)>\max_iw_i$, then $\lct(f)\geq\min\{1,\frac{n}{1+\mu(f)^{1/n}}\}$, with equality if $\mu(f)=(\ord(f)-1)^n$.
    \item\label{cor:lct_tau} If $n=2$, then $\lct(f)\leq\frac12+\frac{1}{\tau(f)+1}$, with equality for $f=x^2+y^{\tau(f)+1}$.
\end{enumerate}
\end{corollary}

\begin{proof}[Proof of Part~\ref{cor:lct_milnor} of Corollary~\ref{corB}]
Since $\mu(f)<\infty$, we also have $\tau(f)<\infty$. Proposition~\ref{prop:tau_p} and Remark~\ref{rem:tau_p} imply that, for every $p\gg0$, $\mu_p(f)=\mu(f)$ and $\tau_p(f)=\tau(f)$. Therefore, Proposition~\ref{prop:milnor_tjurina} gives
\[
    \lct(f)\geq\fpt(f_p)\geq
    \frac{1}{\tau_p(f)}-\frac{1}{\mu_p(f)}
    =\frac{1}{\tau(f)}-\frac{1}{\mu(f)},
\]
and the result follows.
\end{proof}

\begin{proof}[Proof of Parts~\ref{cor:lct_weighted}, \ref{cor:lct_milnor_min} and~\ref{cor:lct_tau} of Corollary~\ref{corB}]
For $p\gg0$, we have $\operatorname{in}_w(f_p)=\operatorname{in}_w(f)_p$, $\ord_w(f_p)=\ord_w(f)$, $\mu_p(f)=\mu(f)$, $\tau_p(f)=\tau(f)$ and $\mu_p(\operatorname{in}_w(f))=\mu(\operatorname{in}_w(f))$, by Proposition~\ref{prop:tau_p} and Remark~\ref{rem:tau_p}. With $S=\frac{w_1+\cdots+w_n}{\ord_w(f)}$, Propositions~\ref{prop:sqh_bound} and~\ref{prop:weighted_fpt} give $\min\{1,S-\frac{\max\{0,n-1-S\}}{p}\}\leq\fpt(f_p)\leq\min\{1,S\}$, and Proposition~\ref{prop:lct_fpt} gives Part~\ref{cor:lct_weighted} of the corollary. Similarly, Proposition~\ref{prop:sqh_milnor} gives $\fpt(f_p)\geq\lambda-\frac{n-1-\lambda}{p}$ for $p\gg0$, so $\lct(f)\geq\lambda$. If $\mu(f)=(d-1)^n$ with $d=\ord(f)$, then for $p\gg0$ also $\mu(f_p)=(d-1)^n$ and $\ord(f_p)=d$. So Lemma~\ref{lem:milnor_min} gives $\mu((f_p)_d)<\infty$, and Proposition~\ref{prop:sqh_bound} with $w=(1,\dots,1)$, together with Proposition~\ref{prop:weighted_fpt}, gives $\min\{1,\frac nd-\frac{(n-1)d-n}{pd}\}\leq\fpt(f_p)\leq\min\{1,\frac nd\}$. Letting $p\to\infty$ gives $\lct(f)=\min\{1,\frac nd\}=\lambda$. For $p>\max\{5,\tau(f)\}$, Proposition~\ref{prop:tau_upper} gives $\fpt(f_p)\leq\frac12+\frac{1}{\tau(f)+1}$. Finally, $x^2+y^m$ has $\tau=m-1$, and $\lct=\frac12+\frac1m$ by Part~\ref{cor:lct_weighted} of the corollary.
\end{proof}

\noindent Part~\ref{cor:lct_weighted} of Corollary~\ref{corB} recovers the classical value of the log canonical threshold of a semi-quasi-homogeneous singularity. Part~\ref{cor:lct_milnor_min} improves on $\lct(f)\geq\frac{1}{\ord(f)}$ whenever $\mu(f)<(n\ord(f)-1)^n$, and Conjecture~\ref{conj:milnor} would remove its hypothesis on $\mu(f)$. For $n=2$, this follows from Theorem~\ref{thm:conj_cases}: every $f\in\mathfrak{m}^2\subset\mathbb{Z}[[x,y]]$ with $\mu(f)<\infty$ satisfies $\lct(f)\geq\min\{1,\frac{2}{1+\sqrt{\mu(f)}}\}$.

\begin{proof}[Proof of Corollary~\ref{cor:MTW}]
The proofs of Lemmas~\ref{lem:hensel} and~\ref{lem:branch_face} work over any algebraically closed field, and the changes of coordinates in the proof of Lemma~\ref{lem:branch_face} are a linear change of coordinates followed by finitely many substitutions $y\mapsto y+\lambda x^c$. Applying that proof over $\overline{\mathbb{Q}}$, we obtain a number field $K$, a polynomial automorphism $\varphi$ of $K[x,y]$ whose inverse is also polynomial, and $F=f\circ\varphi\in K[x,y]$ with $\operatorname{in}_w(F)=u(y^a-\lambda x^c)^d$ for $w=(a,c)$, where $2\leq a<c$ are coprime and $u,\lambda\in K^\times$.

Choose a finite set $S$ of primes of $\mathcal{O}_K$ such that the coefficients of $\varphi$, $\varphi^{-1}$ and $F$ are $S$-integral and $u,\lambda$ are $S$-units. For a prime $\mathfrak{p}\notin S$ over $p$, reduction modulo $\mathfrak{p}$ gives an automorphism $\varphi_{\mathfrak{p}}$ and $F_{\mathfrak{p}}=f_p\circ\varphi_{\mathfrak{p}}$. So $\fpt(f_p)=\fpt(F_{\mathfrak{p}})$, and $\operatorname{in}_w(F_{\mathfrak{p}})=u_{\mathfrak{p}}(y^a-\lambda_{\mathfrak{p}}x^c)^d$ with $u_{\mathfrak{p}},\lambda_{\mathfrak{p}}\neq0$.

Set $N=ac$. By Proposition~\ref{prop:face_sandwich}, for every prime $p$ lying under no prime of $S$ we have
\[
    \frac{a+c}{acd}-\frac{1}{dp}<\fpt(f_p)\leq\frac{a+c}{acd},
\]
with equality on the right whenever the base-$p$ digits of $\frac1a$ and $\frac1c$ add without carrying, in particular whenever $p\equiv1\pmod N$. Letting $p\to\infty$ and using Proposition~\ref{prop:lct_fpt}, we get $\lct(f)=\frac{a+c}{acd}$. Since Lemma~\ref{lem:face_semigroup} holds over any algebraically closed field, applying it over $\mathbb{C}$ to the irreducible power series $F$ gives $ad=\beta_0$ and $cd=\beta_1$. Hence $d=g$, $\frac{a+c}{acd}=\frac1{\beta_0}+\frac1{\beta_1}$, and $a$, $c$ and $N=ac$ are as in the statement. Finally, Dirichlet's theorem~\cite{apostol1976introduction} gives infinitely many primes $p\equiv1\pmod N$.
\end{proof}

\begin{remark}\label{rem:MTW_curves}
\textup{The same proof applies to polynomials with coefficients in a number field. By the prime number theorem for arithmetic progressions~\cite{apostol1976introduction}, the primes $p$ with $\fpt(f_p)=\lct(f)$ have relative lower density at least $1/\phi(ac)$, where $\phi$ is Euler's totient function.}
\end{remark}

\begin{remark}\label{rem:MTW_open_general}
\begin{enumerate}
    \item \textup{The question of whether for every $f\in\mathbb{Z}[\underline{x}]$ there are infinitely many primes $p$ such that $\fpt(f_p)=\lct(f)$ is a fundamental conjecture in the study of $F$-pure thresholds; see Conjecture 3.6 in~\cite{mustata2005f}. It is known in some special cases (see, for example,~\cite{hernandez2016f,hernandez2011f,fedder1983F,smith1997vanishing,ein2007invariants}), and by Corollary~\ref{cor:MTW} it holds for irreducible plane curves, but it remains open in general.}
    \item \textup{Analogous results hold for a polynomial $f\in\mathbb{C}[\underline{x}]$ after spreading out its coefficients over a finitely generated $\mathbb{Z}$-algebra $A$. For every maximal ideal $\mathfrak n$ of $A$, the residue field $\frac{A}{\mathfrak n}$ is finite, hence isomorphic to $\mathbb{F}_{p^r}$ for some prime $p$ and some $r>0$. For more information, see Section 1 in Chapter 6 of~\cite{schwede2024singularities}. For isolated singularities, the determinacy theorem (Proposition~\ref{prop:mt_basic}) allows us to replace the relevant power series by contact-equivalent polynomials.}
\end{enumerate}
\end{remark}

Finally, we recall the Brian{\c{c}}on--Skoda exponent.

\begin{definition}\label{def:eBS}
Given $f\in\mathfrak{m}\subset\mathbb{C}[[\underline{x}]]$ such that $f\in\sqrt{\Jac(f)}$, the \textbf{Brian{\c{c}}on--Skoda exponent} of $f$, denoted $e^{\textup{BS}}(f)$, is the smallest integer $N\geq1$ such that $f^N\in\Jac(f)$.
\end{definition}

\begin{remark}
\textup{The name Brian{\c{c}}on--Skoda comes from the Brian{\c{c}}on--Skoda theorem (see~\cite{skoda1974cloture}), which gives $e^{\textup{BS}}(f)\leq n$ for every $f\in\mathfrak{m}\subset\mathbb{C}[[\underline{x}]]$, where $n$ is the number of variables. In positive characteristic, the analogous conclusion $\nu_0(f,\Jac(f))+1\leq n$ holds under the stronger hypothesis $f\in\overline{\Jac(f)}$, where the bar denotes integral closure (see Remark 4.5 in~\cite{hefez2019hypersurface}). Radical membership alone does not suffice: for odd $p$, the series $f=x^p+x^{rp+1}+y^2$ has $\Jac(f)=\langle x^{rp},y\rangle$ and $\nu_0(f,\Jac(f))+1=r$, for every $r\geq1$. The Brian{\c{c}}on--Skoda exponent has been studied in the context of topology and complex isolated singularities, for example, in~\cite{jung2022brianccon,almiron2023tjurina}.}
\end{remark}

\begin{remark}\label{rem:ebs}
\textup{Let $f\in\mathfrak{m}^2\subset\mathbb{C}[[\underline{x}]]$ with $\mu(f)<\infty$. Then $\lct(f)\mu(f)+1\geq e^{\textup{BS}}(f)$. Indeed, if $d=\ord(f)$ and $M=\mu(f)$, then $f^{\lceil M/d\rceil}\in\mathfrak m^M\subseteq\Jac(f)$ by Proposition~\ref{prop:mt_basic}, so $e^{\textup{BS}}(f)-1<\frac Md\leq M\lct(f)$, as $\lct(f)\geq\frac1d$.}
\end{remark}

\end{document}